\documentclass[a4paper,11pt]{amsart}
\newif\ifcomments

\usepackage{amsmath, amsthm, amsfonts,stmaryrd,amssymb}
\usepackage{accents}
\usepackage{tikz-cd}
\tikzcdset{row sep/normal = 1cm}
\tikzcdset{column sep/normal = 1cm}
\usepackage{mathrsfs}                      
\usepackage{mathtools}                     
\usepackage{booktabs}
\usepackage{todonotes}
\usepackage[all]{xy}
\usepackage[utf8]{inputenc}
\usepackage{mathtools}                     
\usepackage{todonotes} 
\usepackage{bbm}                         
\usepackage{framed}
\usepackage{graphicx}   
\usepackage[margin=3.5cm]{geometry}
\usepackage{hyperref} 
\usepackage{xcolor}
\usepackage[abs]{overpic}
\hypersetup{
    colorlinks=true,
    linkcolor=blue,  
    urlcolor=cyan,
    pdftitle={},
    pdfauthor={},
} 
\usepackage{theoremref}
\usepackage{subfig}

\numberwithin{equation}{section}
\theoremstyle{plain}
\newtheorem{theorem}{Theorem}[section]

\newtheorem{remark}[theorem]{Remark}
\newtheorem{lemma}[theorem]{Lemma}

\newtheorem{proposition}[theorem]{Proposition}

\theoremstyle{definition}

\theoremstyle{remark}

\def\ed{\mathrm{d}}

\def\exp{\operatorname{exp}}

\let\bs=\boldsymbol

\let\mc=\mathcal

\begin{document}
\title[Gradient flows of interacting cells as discrete porous media flows]{Gradient flows of interacting Laguerre cells as discrete porous media flows}
\author{Andrea Natale}
\address{Andrea Natale (\href{mailto:andrea.natale@inria.fr}{\tt andrea.natale@inria.fr}), Inria, Univ. Lille, CNRS, UMR 8524 - Laboratoire Paul Painlevé, F-59000 Lille, France
} 
\date{\today}
\maketitle

\begin{abstract}
We study a class of discrete models in which a collection of particles evolves in time following the gradient flow of an energy depending on the cell areas of an associated Laguerre (i.e.\ a weighted Voronoi) tessellation.
We consider the high number of cell limit of such systems and, using a modulated energy argument, we prove  convergence towards smooth solutions of nonlinear diffusion PDEs of porous medium type.  
\end{abstract}

\section{Introduction}

Voronoi and Laguerre tessellations are a popular tool to describe the neighborhood relations within particle systems, and therefore to model the particle interactions and dynamics. For instance, they have been used as a model for biological cells \cite{honda1983geometrical,bock2010generalized}, to describe the regions of influence of different agents in territorial models in ecology  \cite{votel2009equilibrium}, or also as discretization tools in continuum mechanics and fluid dynamics \cite{gallouet2022convergence,leclerc2020lagrangian}. 
This article focuses on a specific class of models
in which the particles evolution in space is governed by the gradient flow of an energy depending on a Laguerre decomposition of a given domain. Our primary interest is to investigate the high number of particles limit of these models, and show how to interpret the particle dynamics as a discrete version of porous media flow, reproducing its Lagrangian gradient flow structure \cite{evans2005diffeomorphisms}. Adopting this point of view, we will establish quantitative estimates for the convergence of the discrete models to their continous counterparts.

\subsection{Problem description} Given $N\in \mathbb{N}$, a tessellation $\mc{L}=  \{L_1,\ldots, L_N\}$ of a measurable bounded set $A\subset \mathbb{R}^d$ is a collection of a finite number of measurable subsets $L_i\subseteq A$, called cells,  such that 
\[
\mathrm{int}(L_i) \cap \mathrm{int}(L_j)= \varnothing \quad \forall \, i\neq j\,, \quad \text{and} \quad \bigcup_{i} L_i = A\,. \]
We call $\mathbb{T}_N(A)$ the set of tessellations of $A$ composed of $N$ cells, and 
\[
\mathbb{T}_N^s(A) \coloneqq \{ \mc{L} \,:\, \exists\, B\subseteq A \text{ such that }  \mc{L}\in \mathbb{T}_N(B)\}\,.
\]
Given a compact domain $\Omega\subset \mathbb{R}^d$ with Lipschitz boundary, we study the dynamics of $N$ interacting cells, represented by a tessellation in $\mathbb{T}_N^s(\Omega)$, and whose location is parameterized by a vector of cell centers (or particles) $X = (x_1,\ldots,x_N)\in {(\mathbb{R}^d)}^N$. 
Specifically, for a given subset of admissible tessellations $\mathbb{L}_N(\Omega) \subseteq \mathbb{T}_N^s(\Omega)$ and a fixed parameter $\varepsilon>0$, we consider the following energy:
\begin{equation}\label{eq:FepsX}
F_\varepsilon(X) \coloneqq \inf_{\mc{L} \in\mathbb{L}_N(\Omega)} \sum_{i} \int_{L_i}\frac{|x-x_i|^2}{2\varepsilon} \, \ed x  + \sum_i C_i(|L_i|)\,.
\end{equation}
In practice, we will focus on the two cases where $\mathbb{L}_N(\Omega)$ is either $\mathbb{T}_N(\Omega)$ (the union of the cells is fixed) or $\mathbb{T}_N^s(\Omega)$ (the union of the cells is is not determined a priori, but it is rather the set that yields the minimal value of problem \eqref{eq:FepsX}). 
Loosely speaking, the first term in \eqref{eq:FepsX} measures how  closely the tessellation approximates the particle distribution. The second term is the energy of the tessellation which we suppose to depend only on the cell volumes, $C_i:[0,\infty)\rightarrow \mathbb{R}$ being a given function which might be different for each cell. For example, a common choice for $C_i$ used to model biological cells, is the (non)linear spring model 
\[
C_i(a) = C(a) = K(a) \frac{ |a - \bar{a}|^2}{2}\,  \quad \forall\, i\,,
\]
where $K(a)>0$ is a scaling factor called bulk modulus and which might depend on $a$, and $\bar{a} = |\Omega|/N$ is the target volume which is assumed common to all cells.

The dynamics of the cell centers on the time interval $[0,T]$ is governed by the gradient flow of $F_\varepsilon$, with respect to a weighted 
$l^2$ metric on $(\mathbb{R}^d)^N$, defined as follows:
\begin{equation}\label{eq:inner}
\langle \dot{X},\dot{Y} \rangle_{m^0} \coloneqq \sum_i m_i^0 \langle \dot{x}_i , \dot{y}_i \rangle
\end{equation}
for all $\dot{X},\dot{Y} \in (\mathbb{R}^d)^N$ and where $m^0 \coloneqq \{m_1^0,\ldots, m^0_N\} \in \mathbb{R}^N_{>0}$. More precisely the evolution of the cell centers is given by a curve $X:[0,T] \rightarrow (\mathbb{R}^d)^N$ satisfying 
  \begin{equation}\label{eq:gfdiscrete}
 \dot{X}(t) = - \nabla_{m^0} F_\varepsilon(X(t)) 
 \end{equation}
for all $t\in(0,T)$, with a given initial condition $X(0)= X^0 =(x^0_i) \in (\mathbb{R}^d)^N$, where $\nabla_{m^0}$ denotes the gradient with respect to \eqref{eq:inner}.
 
\subsection{Relation with Laguerre tessellations and other discrete models} A Laguerre tessellation (also called power diagram or weighted Voronoi tessellation) is a tessellation of the domain $\mc{L}(X,w) =\{ L_i(X,w)\}_i \in \mathbb{T}_N(\Omega)$ parameterized by a set of particles $X =(x_1,\ldots,x_N)\in (\mathbb{R}^d)^N$ and associated weights $w=(w_1,\ldots, w_N)\in \mathbb{R}^N$, and in which the cells $L_i(X,w)\subset \Omega$ are defined as follows 
\[
L_i(X,w) \coloneqq \{ x \in \Omega\,:\, |x-x_i|^2  - w_i \leq |x - x_j|^2 - w_j\, \quad \forall j \neq i \}\,.
\] 
We will also refer to $x_i$ as the cell center of the cell $L_i(X,w)$. Note that if $w$ is a constant vector then we retrieve the standard Voronoi tessellation of $\Omega$. 

The Voronoi tessellation can be shown to be the unique minimizer of the energy \eqref{eq:FepsX} when $C_i =0$ and $\mathbb{L}_N(\Omega) = \mathbb{T}_N(\Omega)$. In this case, the remaining term in \eqref{eq:FepsX} is sometimes referred to as centroidal Voronoi tessellation energy \cite{du1999centroidal,liu2009centroidal}, and the resulting model governed by \eqref{eq:gfdiscrete} coincides with the Voronoi liquid described in \cite{ruscher2016voronoi}, which is a fluid dynamic interpretation of the Lloyd's algorithm \cite{lloyd1982least,merigot2021non}.

In general,  the tessellations that solves the minimization problem \eqref{eq:FepsX} with $\mathbb{L}_N(\Omega) = \mathbb{T}_N(\Omega)$ is always a Laguerre tessellation, and when $\mathbb{L}_N(\Omega) = \mathbb{T}^s_N(\Omega)$ is contained in one (in particular each cell of the optimal tessellation is the intersection of a Laguerre cell with a ball); see Section \ref{sec:energy}. Our discrete model is therefore related to cell evolution models based on Laguerre tessellations \cite{jones2012modeling}. In Voronoi cell models, for example, one imposes the tessellation to be Voronoi and the energy of the system only contains the second term in \eqref{eq:FepsX} (plus extra energy terms often related to the cell perimeter or nodes distance). While some studies have focused on analyzing specific features of these discrete flows (e.g., their long term behaviour \cite{elsey2015mean}), estabilishing continuous limits for such models is not trivial, partly because the energies considered are usually more complex than those we treat here. Available results are therefore limited to 1d \cite{fozard2010continuum} or formal calculations \cite{alt2003nonlinear}. In this light, the energy \eqref{eq:FepsX} leads to a modified dynamics which is however more amenable to theoretical analysis, at least for the case of energies only depending on the cell area.

\subsection{Relation with Lagrangian discretizations of porous media flow}\label{sec:relationporous} In the following, we focus on the case where
\begin{equation}\label{eq:CU}
C_i(a) \coloneqq \left\{ 
\begin{array}{ll} 
\displaystyle U\left(\frac{m_i^0}{a}\right) a & \text{if } a>0\,,\\
+\infty & \text{otherwise}\,,
\end{array}\right.
\end{equation}
where $U:[0,\infty) \rightarrow \mathbb{R}$ is a smooth strictly convex function with superlinear growth, with $U(0)=0$. For this energy, the particle dynamics generated by \eqref{eq:gfdiscrete} can be reinterpreted as a spatially discrete version of the Lagrangian formulation of the porous medium equation, describing the evolution of a density $\rho:[0,T]\times \Omega \rightarrow [0,\infty)$ as the solution of the PDE:
\begin{equation}\label{eq:pde0}
\left\{
\def\arraystretch{1.5}
\begin{array}{ll}
\partial_t \rho - \mathrm{div}\left[ \rho \nabla U'(\rho)\right]= 0 & \text{ on }(0,T) \times \Omega \,,\\
\nabla U'(\rho) \cdot n_{\partial \Omega} = 0 &\text{ on }(0,T) \times \partial \Omega\,,
\end{array}
\right.
\end{equation} 
where $n_{\partial \Omega}$ denotes the unit normal to the boundary $\partial \Omega$, with given initial conditions $\rho(0,\cdot) = \rho^0$. 

To make this precise, let us fix a smooth strictly-positive reference density $\nu:\Omega \rightarrow (0,\infty)$, and consider the energy $\mc{F}: \mathrm{Diff}(\Omega) \rightarrow \mathbb{R}$ on the space of diffeomorphisms of $\Omega$, defined by
\begin{equation}\label{eq:Fphi}
\mc{F}(\varphi) = \int_\Omega U\left(\frac{\nu}{\mathrm{det}(\nabla \varphi)}\right)\circ \varphi^{-1}  \ed x\,.
\end{equation}
Given $\Phi\in\mathrm{Diff}(\Omega)$ such that 
\begin{equation}\label{eq:rho0push}
\rho^0 = \frac{\nu}{\mathrm{det}(\nabla \Phi)} \circ \Phi^{-1}\,,
\end{equation}
one can check that, at least formally, the flow $\varphi: [0,T] \rightarrow \mathrm{Diff}(\Omega)$ of the vector field $-\nabla U'(\rho)$ solves the gradient flow system
\begin{equation}\label{eq:gflag}
\left\{ \begin{array}{l}
\partial_t \varphi(t) = - \nabla_{\mathbb{F}}\mc{F}(\varphi(t))\\
\varphi(0) = \Phi 
\end{array}
\right.
\text{ and }\quad
\rho(t) = \frac{\nu}{\det \nabla \varphi(t)}\circ \varphi(t)^{-1}\,,
\end{equation}
where $\mathbb{F} \coloneqq L^2_\nu(\Omega;\mathbb{R}^d)$ and therefore $\nabla_{\mathbb{F}}$ is the gradient computed with respect to the $L^2$ inner product weighted by $\nu$ (see \cite{evans2005diffeomorphisms} for details on this interpretation, or also Appendix \ref{sec:applag}).

Let us now fix a reference tessellation $\mc{T}_N = \{T_i\}_i \in \mathbb{T}_N(\Omega)$, and let $\mathbb{F}_N$ be the space of piecewise constant functions on the tessellation with values in $\mathbb{R}^d$, i.e.\,
\[
\mathbb{F}_N \coloneqq \{\varphi^X \in \mathbb{F}: \varphi^X(x) = x_i \in \mathbb{R}^d ~\text{  for a.e.\ } x \in T_i \}\,.
\]
Any element $\varphi^X\in \mathbb{F}_N$ can be identified with a collection of particles $X=(x_i)_i\in(\mathbb{R}^d)^N$ given by the collection of images of the cells $T_i$. Then, a general strategy to construct a particle discretization of equation \eqref{eq:gflag} is to look for solutions $\varphi^X:[0,T]\rightarrow \mathbb{F}_N$ of 
\begin{equation}\label{eq:gfphix}
\partial_t \varphi^X(t) = -\nabla_{\mathbb{F}_N}\tilde{\mc{F}}(\varphi^X(t))\,,
\end{equation}
where $\tilde{\mc{F}}:\mathbb{F}_N\rightarrow \mathbb{R}$ is an appropriate discrete version of $\mc{F}$, and with $\varphi^X(0)$ being an approximation of $\Phi$ in $\mathbb{F}_N$. Different choices of $\tilde{\mc{F}}$ lead to different methods. We refer to \cite{carrillo2021lagrangian} for a review of possible strategies in this context (see in particular \cite{carrillo2017numerical}, for an approach that is particularly close to the one we study in this article). Importantly, our discrete model \eqref{eq:gfdiscrete} is equivalent to \eqref{eq:gfphix} for an appropriate variational regularization of the energy (see again Appendix \ref{sec:applag} for details). Such a regularization is related to recent approaches based on semi-discrete optimal transport recalled in Section \ref{sec:semidiscrete}.

\subsection{Relation with semi-discrete optimal transport} \label{sec:semidiscrete}
The energy in \eqref{eq:FepsX} admits a reformulation based on semi-discrete optimal transport \cite{merigot2021optimal}. In fact, denoting by $W_2(\mu_1,\mu_2)$ the $L^2$ Wasserstein distance between two positive measures $\mu_1,\mu_2\in \mc{M}_+(\Omega)$, one can show that if $x_i\neq x_j$ for all $i\neq j$,
\begin{equation}\label{eq:FepsXw2}
F_\varepsilon(X) = \inf_{a\in \mathbb{R}^N_{> 0},\eta \in \mc{C}} \frac{W_2^2 \Big( \sum_{i} a_i\delta_{x_i}, \eta \Big)}{2\varepsilon} + \sum_i U\left(\frac{m^0_i}{a_i}\right){a_i}\,,
\end{equation}
where we used for $C_i$ the expression given in \eqref{eq:CU}, and where $\mc{C}$ is a convex subset of $\mc{M}_+(\Omega)$: in particular, $\mathcal{C} = \{\mathrm{Leb}\}$ (where $\mathrm{Leb}$ is the Lebesgue measure on $\Omega$) if $\mathbb{L}_N(\Omega)= \mathbb{T}_N(\Omega)$, and $\mathcal{C} = \{ f\ed x\,:\, f:\Omega\rightarrow [0,1] \}$ if $\mathbb{L}_N(\Omega) = \mathbb{T}_N^s(\Omega)$ (see Appendix \ref{sec:otapp} for a proof).

The functional in \eqref{eq:FepsXw2}, before minimization over $a$, has already appeared in the literature, for applications in material science or optimal planning \cite{bourne2015centroidal, buttazzo2009mass}, but with a focus on studying or computing its minimizers.

In this work, we rather regard \eqref{eq:FepsXw2} as an approximation of \eqref{eq:Fphi}, and system \eqref{eq:ode} as a deterministic particle discretization of \eqref{eq:pde0}. The idea of using deterministic particle methods to discretize diffusion models was already put forward by Russo in \cite{russo1990deterministic}.    
The approach of constructing such methods by regularizing the energy via a variational problem of the same type of \eqref{eq:FepsXw2}, however, was proposed by Brenier in \cite{brenier2000derivation} for discretizing the incompressible Euler equations, and later developed using semi-discrete optimal transport tools in \cite{merigot2016minimal,gallouet2018lagrangian}. In \cite{gallouet2022convergence,leclerc2020lagrangian} these methods were further developed to discretize the same nonlinear diffusion models considered in this article as well as the compressible (barotropic) Euler equations,
and in  \cite{sarrazin2022lagrangian} in the context of mean field games. In these works, one regards the energy of the system \eqref{eq:Fphi} as a function of the density, and considers its Moreau-Yosida regularizations on the space of probability measures $\mc{P}(\Omega)$ with respect to the $W_2$ distance. This yields a very similar expression to \eqref{eq:FepsX}:
\begin{equation}\label{eq:ftilde}
	\tilde{F}_\varepsilon(X) \coloneqq \inf_{\rho \in \mc{P}^{ac} (\Omega)} \frac{W_2^2 \Big( \sum_i m^0_i \delta_{x_i}, \rho \Big)}{2\varepsilon} + \int_\Omega U(\rho)\,,
\end{equation}
where $\mc{P}^{ac}(\Omega)$ is the set of absolutely continuous probability measures on $\Omega$.

Note that the advantage of using \eqref{eq:FepsXw2} with respect to \eqref{eq:ftilde} is that the first implies a piece-wise constant density reconstruction (see equation \eqref{eq:muNfun} below) which is easier to handle numerically than the minimizers of \eqref{eq:ftilde} whose structure strongly depends on $U$. Moreover the variational definition \eqref{eq:ftilde} is also easier to generalize to more complex energies (see Remark \ref{rem:generalizations}). Finally, we remark that another discretization strategy similar to ours, which also leads to piece-wise constant densities on Laguerre cells, was proposed in \cite{benamou2016discretization}, but the aim of the latter work was not to derive quantitative convergence estimates as we do here.

\subsection{Continuous limit} 

For a given initial condition $X^0$ such that $x^0_i\neq x^0_j$ for all $i\neq j$, let $t\in[0,T] \mapsto X(t)$ solve system \eqref{eq:gfdiscrete} with energy \eqref{eq:CU} and 
\[
m^0_i = \int_{T_i} \nu = \int_{\Phi(T_i)} \rho^0 \,,
\]
where $\rho^0$ satisfies \eqref{eq:rho0push} as before. 
Consider the discrete density
\begin{equation}\label{eq:muNfun}
\bar{\mu}_N(t,\cdot) \coloneqq \sum_i \frac{m^0_i}{|L_i(t)|} \bs{1}_{L_i(t)} \,,
\end{equation}
where $L_i(t)$ is the unique (see Section \ref{sec:dualformulation}) optimal tessellation for problem \eqref{eq:FepsX} associated with the positions $X(t)$, and where $\bs{1}_{L_i(t)}$ is the  characteristic function of the set $L_i(t)$. {Let $\varphi^X_N: [0,T] \times \Omega \rightarrow \mathbb{R}^d$ be the map defined by
\begin{equation}\label{eq:phiNX}
\varphi^X_N(t,x) = X_i(t) \,, \quad \text{for a.e. } x\in T_i \,,
\end{equation}
and for all $t\in[0,T]$ and $i=1,\ldots,N$, with $\mc{T}_N = \{T_i\}_i \in \mathbb{T}_N(\Omega)$ being a fixed reference tessellation as in Section \ref{sec:relationporous}.
}
Our main result states that $\bar{\mu}_N$ converges to sufficiently smooth solutions $\rho$ of \eqref{eq:pde0} and $\varphi^X$ converges to the flow of $-\nabla U'(\rho)$, as long as the error in the initial conditions, measured by  
\begin{equation}\label{eq:initialcond0}
\delta_N^2 \coloneqq \sum_i  \int_{T_i} |\Phi(x) - x_i^0|^2 \nu(x) \ed x\,,
\end{equation}
and $\varepsilon$ go to zero with appropriate rates.

\begin{theorem}\label{th:convergence} Let $U:[0,\infty) \rightarrow \mathbb{R}$ be a smooth strictly convex function, with $U(0)=0$,  satisfying the assumptions of Lemma \ref{lem:Abound} and suppose that there exist $R,\alpha>1$ and $\beta>0$, such that 
\[
U(r)- \inf U \geq \beta r^\alpha \quad \forall\, r\geq R\,.
\]
Suppose that $\rho:[0,T]\times \Omega \rightarrow [0,\infty)$ is a strong  solution of \eqref{eq:pde0}, such that $\rho^0 : \Omega \rightarrow [\rho_{min},\infty)$ with $\rho_{min}>0$ is of class $C^{1,1}$ and  satisfies \eqref{eq:rho0push}, and $\nabla U'(\rho)$ is of class $C^{2,1}$ in space, uniformly in time. { Let $\varphi:[0,T]\times \Omega \rightarrow \Omega$ be the flow of $-\nabla U'(\rho)$,  satisfying for all $t\in(0,T)$ and $x\in\Omega$,\[
\partial_t \varphi(t,x) = -\nabla U'(\rho(t,\varphi(t,x)))\,, \quad \varphi(0,x) = \Phi(x)\,.
\]}%
Moreover, let $\bar{\mu}_N:[0,T]\times \Omega \rightarrow [0,\infty)$ { and $\varphi_N^X:[0,T]\times \Omega\rightarrow \mathbb{R}^d$ } be the maps defined in \eqref{eq:muNfun} and \eqref{eq:phiNX} via system \eqref{eq:gfdiscrete}, with energy \eqref{eq:CU}, 
and suppose that $C_0^{-1}/ N \leq m^0_i \leq C_0/N$ for a constant $C_0>0$ and for all $1\leq i\leq N$. 
Then
\begin{equation}\label{eq:est}
\max_{t\in[0,T]}\left\{ { \|\varphi_N^X(t,\cdot) -\varphi(t,\cdot) \|^2_{L^2(\nu)} }+\int_\Omega U(\bar{\mu}_N(t,\cdot)|\rho(t,\cdot))\right\} \leq  C \left( \frac{\delta_N^2}{\varepsilon} + \varepsilon^{p-1} \right)\,, 
\end{equation}
where $p = \min\{2,\alpha\}$, 
$U(r|s) \coloneqq  U(r) - U(s) - U'(s)(r-s)$ for all $r\geq 0$ and $s>0$, and where the constant $C>0$ only depends on  $\sup_{t\in[0,T]}\|\nabla U'(\rho(t,\cdot)\|_{C^{2,1}}$, $\|\rho^0\|_{C^{1,1}}$, $\rho_{min}$, $C_0$, $\mathrm{diam}(\Omega)$, $U$, $T$, and $d$.
\end{theorem}

The proof is contained in Section \ref{sec:convergence}. Just as in \cite{gallouet2022convergence}, it relies on a Gr\"onwall argument applied on an appropriately constructed modulated energy (as in the classical approach to obtain weak strong stability results for \eqref{eq:pde0}; see, e.g., Chapter 5 in \cite{dafermos2005hyperbolic}).
Finally, note that if for example we set
\begin{equation}\label{eq:x0}
x^0_i = \mathrm{arg}\min_{y\in\mathbb{R}^d} \int_{T_i} |\Phi(x) - y|^2  \nu(x)\ed x\,,
\end{equation}
then $\delta_N$ is just the $L^2_\nu$ projection error of $\Phi$ onto $\mathbb{F}_N$, the space of piece-wise constant vector fields on $\mc{T}_N$. Hence, denoting by $h_N = \max_i \mathrm{diam}(T_i)$ the largest cell diameter, 
we have $\delta_N \leq \| \nabla \Phi\|_\infty (\rho^0[\Omega])^{1/2} h_N$, where $\rho^0[\Omega]$ is the integral of $\rho^0$ over $\Omega$.

\begin{remark} The error estimate we prove is actually stronger than \eqref{eq:est}, and it is given explicitly in equation \eqref{eq:finalestimate} (see also Section \ref{sec:external} for an extension in the presence of external potentials). { In particular, we stress that if $U$ is strongly convex on the interval $[\rho_{min}, \infty)$, which is the case for power laws $U(r) = r^\gamma/(\gamma-1)$ with $\gamma \geq 2$ for example, there exists a constant $\lambda>0$ only depending on $\rho$ such that
\[
\lambda \|\bar{\mu}_N -\rho\|_{L^2}^2 \leq \int_\Omega U(\bar{\mu}_N(t,\cdot)|\rho(t,\cdot))\,,
\]
so that equation \eqref{eq:est} can be read as a more standard approximation result for the discrete solution $\bar{\mu}_N$.}
\end{remark}

\section{Analysis of the discrete model}\label{sec:energy}

In this section, we describe in more detail the discrete model \eqref{eq:gfdiscrete}, in particular we explain the link with Laguerre tessellations and provide an explicit formula for the gradient of the energy with respect to the particle positions. Most of the analysis follows the same lines as in \cite{leclerc2020lagrangian}, or uses standard arguments from semi-discrete optimal transport (see Appendix \ref{sec:otapp} or \cite{merigot2021optimal}, for example).

\subsection{Internal energy} We suppose that the functions $C_i$ defining the energy of the tessellation are given by \eqref{eq:CU}
with $U:[0,\infty) \rightarrow \mathbb{R}$ being a smooth strictly convex function with superlinear growth, satisfying $U(0)=0$. 
From this, it is easy to deduce that $C_i$ is strictly convex, decreasing and \[
C_i(a)\rightarrow m^0_i U'(0^+) \quad \text{ as } \quad a\rightarrow +\infty\,.\] As a consequence, { the Legendre transform of $C_i$, denoted $C^*_i:\mathbb{R}\rightarrow (-\infty,\infty]$ and defined by
\begin{equation}\label{eq:legendre}
C^*_i(w) \coloneqq \sup_{a> 0} \left\{w a - C_i(a)\right\}\,,
\end{equation}
}%
is also strictly convex on its effective domain $\mathrm{dom}(C^*_i)\coloneqq\{ w\in\mathbb{R}~:~ C^*_i(w)<\infty\}$. More precisely, $\mathrm{dom}(C^*_i) \subseteq (-\infty,0]$, $C^*_i$ is an increasing diffeomorphism between $(-\infty, 0)$ and $(-\infty,-m_i^0 U'(0^+))$, and $(C^*_i)'(0^-) = +\infty$ (since $C_i$ is finite and decreasing on $(0,\infty)$). 

The pressure function associated with $U$ is the strictly increasing function $P:[0,\infty) \rightarrow [0,\infty)$, defined by 
\begin{equation}\label{eq:thermodynamic}
P(r) = \left\{ 
\begin{array}{ll}
r U'(r) - U(r) & \text{ if } r>0\,,\\
0 & \text{ if } r=0\,,
\end{array}
\right. 
\end{equation}
and satisfying $P'(r)= rU''(r)$ for all $r>0$. This is related to $C_i'$ and $C^*_i$ by
\begin{equation}\label{eq:cstar}
C'_i(a) = - P\left(\frac{m_i^0}{a}\right) \quad\text{and}\quad (C^*_i)'(w) = \frac{m_i^0}{P^{-1}(-w)}\,,
\end{equation}
for all $a>0$ and $w<0$. { The first equation in \eqref{eq:cstar} can be found by direct computation, whereas the second is a consequence of the latter and of the classical relation $((C^*_i)'\circ C_i')(a) = a$ for all $a>0$.}
 
\subsection{Dual formulation}\label{sec:dualformulation} From now on we suppose that $X=(x_1,\ldots,x_N)\in (\mathbb{R}^d)^N$ is given and that $x_i \neq x_j$ for all $i\neq j$. First we rewrite the energy \eqref{eq:FepsX} as follows:
\begin{equation}\label{eq:infsup}
F_\varepsilon(X) = \inf_{\mc{L} \in\mathbb{L}_N(\Omega),a\in \mathbb{R}^N} \sup_{w\in \mathbb{R}^N} \sum_{i} \int_{L_i}\frac{|x-x_i|^2}{2\varepsilon} \, \ed x  + C_i(a_i) + \frac{w_i}{2\varepsilon}(a_i-|L_i|)\,.
\end{equation}
We obtain the dual problem by swapping the inf and the sup,
\begin{equation}\label{eq:dualproblem}
\begin{aligned}
D_\varepsilon(X) & \coloneqq \sup_{w\in \mathbb{R}^N} \inf_{\mc{L} \in\mathbb{L}_N(\Omega),a\in \mathbb{R}^N} \sum_{i} \int_{L_i}\frac{|x-x_i|^2}{2\varepsilon} \, \ed x  + C_i(a_i) + \frac{w_i}{2\varepsilon}(a_i-|L_i|)\\
& = \sup_{w\in \mathbb{R}^N} \inf_{\mc{L} \in\mathbb{L}_N(\Omega)} \sum_{i} \int_{L_i}\frac{|x-x_i|^2-w_i}{2\varepsilon} \, \ed x  - C_i^*\left( -\frac{w_i}{2\varepsilon}\right) \,.
\end{aligned}
\end{equation}
{Note that the second equality in \eqref{eq:dualproblem} just follows from the definition of the Legendre transform $C_i^*$ in \eqref{eq:legendre} and the fact that we can compute the infimum with respect to each component $a_i$ separately. Note also that since we swapped $\inf$ and $\sup$, at the momement we only know that
\begin{equation}\label{eq:orderdual}
F_\varepsilon(X)\geq D_\varepsilon(X)\,,
\end{equation}
but we will show at the end of this section that in fact the equality holds.
}

Let $\phi(w; \cdot):\Omega \rightarrow \mathbb{R}$ be the function defined by $\phi(w;x) = \mathrm{min}_i |x-x_i|^2-w_i$. If $\mathbb{L}_N(\Omega) = \mathbb{T}_N^s(\Omega)$, then $\mc{L} \in \mathbb{T}_N(\tilde{\Omega})$ for some $\tilde{\Omega} \subseteq \Omega$, and
\begin{equation}\label{eq:optimalitylaguerreb}
\begin{aligned}
\sum_{i} \int_{L_i}\frac{|x-x_i|^2-w_i}{2\varepsilon} \ed x & \geq \frac{1}{2\varepsilon} \int_{\tilde{\Omega}} \phi(w;x) \,\ed x \\& \geq \frac{1}{2\varepsilon} \int_{\Omega} \min(0, \phi(w;x)) \,\ed x \\ &
=\sum_{i} \int_{L^s_i(X,w)} \frac{|x-x_i|^2-w_i}{2\varepsilon} \ed x\,,
\end{aligned}
\end{equation}
where
\begin{equation}\label{eq:Lsxw}
L^s_i(X,w) \coloneqq L_i(X,w) \cap B(x_i, (w_i^+)^{1/2}) \,,
\end{equation}
where $B(x_i,(w_{i}^+)^{1/2})$ is the closed ball of radius $w_i^{1/2}$ if $w_i\geq 0$, and $0$ otherwise. On the other hand if $\mathbb{L}_N(\Omega) = \mathbb{T}_N(\Omega)$, by the same argument we obtain
\begin{equation}\label{eq:optimalitylaguerre}
\sum_{i} \int_{L_i}\frac{|x-x_i|^2-w_i}{2\varepsilon} \geq \sum_{i} \int_{L_i(X,w)} \frac{|x-x_i|^2-w_i}{2\varepsilon} \ed x\,.
\end{equation}
 Therefore 
\begin{equation}\label{eq:dual}
D_\varepsilon(X) = \sup_{w\in \mathbb{R}^N} D_\varepsilon(X;w) \,,\end{equation} where \[ D_\varepsilon(X;w)\coloneqq \sum_{i} \int_{L_i^*(X,w)}\frac{|x-x_i|^2-w_i}{2\varepsilon} \, \ed x  - C_i^*\left( -\frac{w_i}{2\varepsilon}\right) \,.
\]
and  
\[
L_i^*(X,w) = \left\{
\begin{array}{ll}
 L_i(X,w) & \text{ if  } \mathbb{L}_N = \mathbb{T}_N(\Omega)\,,\\  L_i^s(X,w) &\text{ if  } \mathbb{L}_N = \mathbb{T}_N^s(\Omega)\,. \end{array}
 \right.
 \]
A graphical representation of an optimal tessellation in the case $\mathbb{L}_N = \mathbb{T}_N^s(\Omega)$ is given in Figure \ref{fig:power}.

\begin{figure}
\begin{picture}(170,110)(0,0)
\put(0,0){\includegraphics[scale=.7]{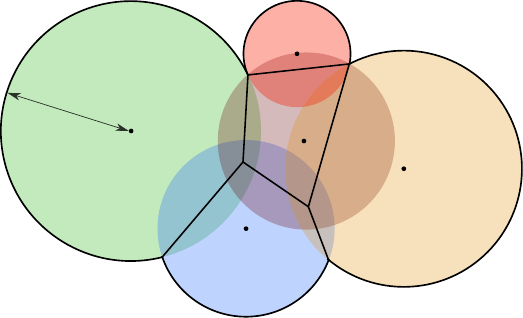}}
\put(18,75){$\sqrt{w_i}$}
\put(49,60){$x_i$}
\end{picture}
\caption{An example of optimal tessellation with cells constructed via equation \eqref{eq:Lsxw}.}\label{fig:power}
\end{figure}
\begin{proposition}\label{prop:optimality} The function $D_\varepsilon(X; \cdot):\mathbb{R}^N \rightarrow (-\infty,\infty]$ is concave and $C^1$ on its effective domain $\mathrm{dom} (D_\varepsilon(X; \cdot)) \subseteq [0,\infty )^N$. In particular for all $w \in (0,\infty)^N$,
\begin{equation}\label{eq:partialder}
\partial_{w_i} D_\varepsilon(X;w) =  \frac{1}{2\varepsilon} (C_i^*)'\left(-\frac{w_i}{2\varepsilon}\right) - \frac{|L_i^*(X,w)|}{2\varepsilon}\, \quad \forall\,i\,.
\end{equation}
Furthermore there exists a unique $w^* \in (0,\infty)^N$ such that $\partial_{w_i} D_\varepsilon(X;w^*) =0 $ for all $i$, or equivalently such that  $|L_i^*(X,w^*)|>0$ and
\begin{equation}\label{eq:opticondw}
P\left(\frac{m^0_i}{|L_i^*(X,w^*)|} \right) = \frac{w_i^*}{2\varepsilon}\, \quad \forall\, i\,,
\end{equation}
and therefore maximising $D_\varepsilon(X;\cdot)$.
\end{proposition}

\begin{proof} The concavity and $C^1$ regularity of $D_\varepsilon(X;\cdot)$ can be proven using standard arguments from the theory of semi-discrete optimal transport \cite{merigot2021optimal}. In particular, consider the function 
\[
Q(w) \coloneqq \sum_i \int_{L_i^*(X,w)}\frac{|x-x_i|^2-w_i}{2\varepsilon} \, \ed x \]
and observe that for any $w,\tilde{w}\in\mathbb{R}^N$
\[
Q(\tilde{w}) \leq \sum_i \int_{L_i^*(X,w)}\frac{|x-x_i|^2-\tilde{w}_i}{2\varepsilon} \, \ed x  = Q(w) + \sum_i \frac{w_i-\tilde{w}_i}{2\varepsilon} |L_i^*(X,w)|\,.
\]
This shows that the super-differential of $Q$ at $w$ is not empty since the vector $(-|L_i^*(X,w)|/(2\varepsilon))_i \in \partial^+Q(w)$, and therefore $Q$ is concave. Furthermore, since $|L_i^*(X,w)|$ is a continuous function of $w$ (see, e.g., Proposition 38 in \cite{merigot2021optimal}), $Q$ (and therefore $D_\varepsilon(X;\cdot)$) is necessarily $C^1$. 
Existence of maximisers can be shown oberving that the function $D_\varepsilon(X;\cdot)$ is coercive since for any $j$,
\[
 \sum_i \int_{L^*_i(X,w_i)} \frac{|x-x_i|^2 - w_i}{2\varepsilon} \ed x \leq \int_\Omega \frac{|x-x_j|^2 - w_j}{2\varepsilon} \ed x \leq \frac{\mathrm{diam}(\Omega)^2}{2\varepsilon}|\Omega| - \frac{|\Omega|}{2\varepsilon} w_j\,.
\]
Since $(C^*_i)'(0^-) =+\infty$ we get that necessarily $w_i>0$, so the optimality conditions hold. Uniqueness of maximizers is a consequence of the strict convexity of the functions $C_i^*$. As a matter of fact, $D_\varepsilon(X;\cdot)$ is strictly concave since it can be expressed as the sum of $Q(\cdot)$, which we proved to be concave, and the function
\[
w\in\mathbb{R}^N \mapsto \sum_i - C^*_i\left( -\frac{w_i}{2\varepsilon} \right)\,,
\]
which is strictly concave on its effective domain.
\end{proof}

From Proposition \ref{prop:optimality}, we can deduce the equivalence with the primal problem \eqref{eq:FepsX}, and the existence and uniquess of  solutions for the latter.{  To see this, recall first that since $\inf \sup \geq \sup \inf$, we always have
$F_\varepsilon(X) \geq D_\varepsilon(X)$ for all $X\in(\mathbb{R}^d)^N$.  
Then, denoting by $w^*$ the unique solution of $\partial_{w_i} D_\varepsilon(X;w^*) =0 $ for all $i$,  by equation \eqref{eq:partialder} we have
\[
(C_i^*)'\left(-\frac{w_i^*}{2\varepsilon}\right) =  |L_i^*(X,w^*)|\,
\]
and therefore, by the Fenchel-Young inequality,
\[
C_i^*\left(-\frac{w_i^*}{2\varepsilon}\right) + C_i(|L_i^*(X,w^*)|) = -\frac{w_i^*}{2\varepsilon} |L_i^*(X,w^*)|\,.
\]
This implies that}%
\[
\begin{aligned}
D_\varepsilon(X) = D_\varepsilon(X;w^*) & = \sum_i \int_{L_i^*(X,w^*)}\frac{|x-x_i|^2}{2\varepsilon} \, \ed x - \frac{w_i^*}{2\varepsilon} |L_i^*(X,w^*)|  - C_i^*\left( -\frac{w_i^*}{2\varepsilon}\right) \\
& = \sum_i \int_{L_i^*(X,w^*)}\frac{|x-x_i|^2}{2\varepsilon} \, \ed x +C_i(|L_i^*(X,w^*)|) \geq  F_\varepsilon(X),
\end{aligned}
\]
{and so $D_\varepsilon(X)=F_\varepsilon(X)$.} Since $w^*$ is the unique maximiser of the dual problem \eqref{eq:dual}, this also shows that the tessellation $\mc{L}^*(X,w^*) \coloneqq\{ L_i^*(X,w^*)\}_i$ is the unique  solution of problem \eqref{eq:FepsX}. { Indeed, supposing that $ \bar{\mc{L}} \coloneqq \{ \bar{L_i}\}_i$ is a minimizer, then 
\begin{equation}\label{eq:uniquenessbound}
\begin{aligned}
F_\varepsilon(X) &= \sum_{i} \int_{\bar{L_i}}\frac{|x-x_i|^2}{2\varepsilon} \, \ed x  + \sum_i C_i(|\bar{L_i}|)\\
& \geq 
\sum_{i} \int_{\bar{L_i}}\frac{|x-x_i|^2 - w_i^*}{2\varepsilon} \, \ed x  - \sum_i C_i^*\left(-\frac{w_i^*}{2\varepsilon}\right)\,,
\end{aligned}
\end{equation}
where we used the Fenchel-Young inequality 
\[
C_i^*\left(-\frac{w_i^*}{2\varepsilon}\right) + C_i(|\bar{L_i}|) \geq -\frac{w_i^*}{2\varepsilon} |\bar{L_i}|\,.
\]
By the same arguments as for \eqref{eq:optimalitylaguerreb} and \eqref{eq:optimalitylaguerre}, one can see that the right-hand side in equation \eqref{eq:uniquenessbound} is strictly larger than $D_\varepsilon(w^*)$ unless $\bar{L_i} =L_i^*(X,w^*)
$ for all $i$, in which case the equality holds. 
}

\begin{remark}\label{rem:ot} Note that setting $C_i(a) = \iota_{\{m_i^0\}}(a)$, the convex indicator function of the set $\{m_i^0\}$, problem \eqref{eq:FepsX} with $\mathbb{L}_N(\Omega) = \mathbb{T}_N(\Omega)$ coincides with a classical optimal transport problem between $\mathbf{1}_\Omega \ed x$ and $\sum_i m_i^0 \delta_{x_i}$. In this case, $C_i^*(w)= m_i^0 w$, and the duality $F_\varepsilon(X) = D_\varepsilon(X)$ reduces to the classical duality theorem from semi-discrete optimal transport \cite{merigot2021optimal}.
\end{remark}

In the following we will need the optimality conditions in Proposition \ref{prop:optimality} in the following alternative form:

\begin{lemma} \label{lem:optimalityu}
Let $\mc{L}=\{L_i\}_i$ solve problem \eqref{eq:FepsX}, with $\mathbb{L}_N(\Omega)$ equal to either $\mathbb{T}_N(\Omega)$ or $\mathbb{T}_N^s(\Omega)$. Then for any smooth vector field $u:\Omega \rightarrow \mathbb{R}^d$ with $u\cdot n_{\partial \Omega} = 0$ on $\partial \Omega$,
\[
\sum_i \int_{L_i} \frac{x-x_i}{\varepsilon} \cdot u(x) \, \ed x = \sum_i \int_{L_i}\left[P\left(\frac{m_i^0}{|L_i|}\right)-\frac{|x-x_i|^2}{2\varepsilon} \right] \mathrm{div}u(x)\, \ed x\,.
\]
\end{lemma}
\begin{proof} 
Consider again the function $\phi(w;\cdot):\Omega \rightarrow \mathbb{R}$ defined by $\phi(w;x) = \mathrm{min}_i |x-x_i|^2-w_i$. If $\mathbb{L}_N(\Omega) = \mathbb{T}_N(\Omega)$ we have
\[
\begin{aligned}
\sum_i \int_{L_i} \frac{x-x_i}{\varepsilon} \cdot u(x) \, \ed x & = \int_{\Omega }  \frac{1}{2\varepsilon} \nabla_x \phi(w^*;x) \cdot u(x) \, \ed x \\
& =\sum_i \int_{L_i}\left[ \frac{w_i^*}{2\varepsilon}-\frac{|x-x_i|^2}{2\varepsilon} \right] \mathrm{div}u(x)\, \ed x\,,
\end{aligned}
\]%
with $w^*$ being the unique maximizer of $D_\varepsilon(X;\cdot)$, and we conclude using Proposition \ref{prop:optimality}. If $\mathbb{L}_N(\Omega) = \mathbb{T}_N^s(\Omega)$ we just need to replace $\phi(w;\cdot)$ with $\min\{\phi(w;\cdot),0\}$.
\end{proof}

\subsection{Discrete dynamical system} Let us introduce the set of particle configurations where at least two particles share the same location:
\[\Delta_N = \{ X=(x_1,\ldots,x_N) \in (\mathbb{R}^d)^N \,: \exists \,i, j  ~\text{such that}~ x_i =x_j\,,~ i \neq j\}\,.\]

\begin{proposition}\label{prop:gradient} The function $X \mapsto F_\varepsilon(X)$ is $C^1$ on $(\mathbb{R}^d)^N \setminus \Delta_N$. Moreover, for any $X\in (\mathbb{R}^d)^N \setminus \Delta_N$,
\[
(\nabla_{m^0} F_\varepsilon(X))_i = \frac{|L_i|}{m_i^0} \frac{x_i - b_i}{\varepsilon}\,, \quad \text{where} \quad b_i \coloneqq \frac{1}{|L_i|} \int_{L_i} x\, \ed x\,,
\]
and where $\mc{L} = \{L_i\}_i$ is the unique minimiser of problem \eqref{eq:FepsX}.
\end{proposition}

The proof is a slight adaptation of the arguments used in \cite{leclerc2020lagrangian} to prove an analogous result, and is therefore postponed to Section \ref{sec:proofgradient} in the appendix.

In view of Proposition \ref{prop:gradient}, the particle dynamics is governed by the system of ODEs
\begin{equation}\label{eq:ode}
\dot{x}_i(t) = -\frac{|L_i(t)|}{m_i^0} \frac{x_i(t) -b_i(t)}{\varepsilon}\,,
\end{equation}
where we denote by $L_i(t)$ the optimal tessellation associated with the particle configuration $X(t) = (x_1(t),\ldots,x_N(t)) \in (\mathbb{R}^d)^N$. Note that this can be interpreted as a generalized continuous-time version of Lloyd's algorithm for optimal quantization. Specifically, the classical Lloyd's algorithm \cite{lloyd1982least} can be formally recovered as a forward Euler discretization of \eqref{eq:ode} in the case where $C_i= 0$ (since in this case the tessellation $(L_i(t))_i$ is Voronoi), whereas setting $C_i = \iota_{\{m_i^0\}}$ as in Remark \ref{rem:ot} on recovers in the same way a variant of Lloyd's algorithm for the uniform quantization problem, studied in \cite{merigot2021non}.

By Proposition \ref{prop:gradient} the right-hand side of \eqref{eq:ode} is a continuous function of $X(t)$ on the open set $(\mathbb{R}^d)^N \setminus \Delta_N$, and therefore the system always admits solutions if $X(0)\in (\mathbb{R}^d)^N \setminus \Delta_N$. We now show that such solutions are always defined for all times (again, adapting similar arguments from \cite{leclerc2020lagrangian}). For this, we will need the following lemma:

\begin{lemma}\label{lem:lip} Let $\mc{L}(X,w)$ the Laguerre tessellation of $\Omega$ associated with the position vector $X\in (\mathbb{R}^d)^N$ and the vector of weights $w$. If $|L_i(X,w)|>0$ for all $i=1,\ldots,N$, then for all $i,j =1,\ldots,N$,
\[
|w_i - w_j| \leq 2\mathrm{diam}(\Omega)|x_i -x_j|\,.
\]
\end{lemma}
\begin{proof}
Since $|L_i(X,w)|>0$, there exists $x\in \Omega$ such that 
\[
|x-x_i|^2 -w_i \leq |x-x_j|^2-w_j\,,
\]
for all $j$. Rearranging terms we obtain,
\[
w_j -w_i \leq |x_j|^2 -|x_i|^2 - 2 x\cdot (x_j -x_i) = (x_j+x_i -2x) \cdot (x_j -x_i) \leq 2 \mathrm{diam}(\Omega) |x_j-x_i|\,.
\]
Swapping the role of $i$ and $j$ we get the result.
\end{proof}

Suppose that $X(t)$ solves \eqref{eq:ode} on some interval $[0,t^*)$ with $X(0)\in \Omega^N\setminus \Delta_N$, and consider the stricly decreasing function $R(s)\coloneqq P(1/s)$ for all $s>0$. Then,  by the mean value theorem,
\[
\left|P\left(\frac{m^0_i}{|L_i(t)|}\right)- P\left(\frac{m^0_j}{|L_j(t)|}\right)\right| \geq C_N \left|\frac{|L_i(t)|}{m^0_i} - \frac{|L_j(t)|}{m^0_j}\right|\,,
\]
where $C_N\coloneqq |R'(|\Omega|/\bar{m}^0)|>0$, with $\bar{m}^0 \coloneqq \min_i m^0_i$, and by the optimality conditions in Proposition \ref{prop:optimality} and then Lemma \ref{lem:lip},
\begin{equation}\label{eq:Rlipbound}
\left|\frac{|L_i(t)|}{m^0_i} - \frac{|L_j(t)|}{m^0_j}\right| \leq \frac{|w_i(t)- w_j(t)|}{2\varepsilon  C_N}  \leq  \frac{\mathrm{diam}(\Omega)}{\varepsilon  C_N} |x_i(t)- x_j(t)|\,.
\end{equation}

From these bounds we can deduce a lower bound on the distance between particles. In particular, first observe that integrating
\[
|x-x_i|^2-w_i \leq |x-x_j|^2-w_j
\]
over a non-empty Laguerre cell $L_i(X,w)$ one obtains $2\langle b_i, x_i -x_j\rangle \geq w_j-w_i$ and swapping the role of $i$ and $j$ we deduce that
\begin{equation}\label{eq:monobary} \langle x_i - x_j, b_i-b_j\rangle\geq 0\,.\end{equation}%
Using this fact, and omitting the time dependency of $x_i$, $b_i$ and $L_i$ to simplify the notation, we obtain
\begin{equation}\label{eq:dtxixj}
\begin{aligned}
\frac{\ed}{\ed t} \frac{|x_i - x_j|^2}{2} & =  \left\langle x_i - x_j , -\frac{ x_i - b_i}{\varepsilon}\frac{|L_i|}{m_i^0} + \frac{ x_j - b_j}{\varepsilon}\frac{|L_j|}{m_j^0}\right\rangle \\
&\geq  - \frac{|L_i|}{m_i^0} \frac{|x_i -x_j|^2}{\varepsilon}  - |x_j -b_j| |x_i-x_j| \left| \frac{|L_i|}{m_i^0} -\frac{|L_j|}{m_j^0}   \right| \frac{1}{\varepsilon} \\
&\geq  - \frac{|\Omega|}{\bar{m}^0} \frac{|x_i -x_j|^2}{\varepsilon}  -\frac{\mathrm{diam}(\Omega)^2}{\varepsilon^2  C_N} |x_i- x_j|^2 \,,
\end{aligned}
\end{equation}
 where we added and subtracted $\langle x_i-x_j, x_j-b_j\rangle|L_i|/(\varepsilon m_i^0)$ and then used $\eqref{eq:monobary}$ to pass from the first to the second line, and used \eqref{eq:Rlipbound} to pass from the second to the third line. By a Gr\"onwall inequality, this shows the long time existence of discrete solutions:

\begin{lemma}\label{lem:hull} If $X(0) \in (\mathbb{R}^d)^N\setminus \Delta_N$, then the solutions to \eqref{eq:ode} are defined for all times $t>0$. Moreover, if $X(0) \in \mathrm{conv}(\Omega)^N\setminus \Delta_N$, then $X(t) \in \mathrm{conv}(\Omega)^N\setminus \Delta_N$ for all $t>0$, where $\mathrm{conv}(\Omega)$ denotes the convex hull of $\Omega$.
\end{lemma}
\begin{proof} Denote by $\pi:\mathbb{R}^d\rightarrow \mathrm{conv}(\Omega)$ the Euclidean projection of $x$ onto $\mathrm{conv}(\Omega)$ and by $d^2(x)= |x-\pi(x)|^2$ the square distance of $x$ from $\mathrm{conv}(\Omega)$. Then since $b_i(t) \in \Omega$,
\[
\frac{\ed}{\ed t} \frac{d^2(x)}{2}  = \langle \dot{x}_i, x_i- \pi(x_i)\rangle = -\frac{|L_i|}{\varepsilon m_i^0} \langle x_i-b_i, x_i- \pi(x_i)\rangle\leq - \frac{|\Omega|}{\varepsilon \bar{m}^0} d^2(x),
\]
with $\bar{m}^0 \coloneqq \min_i m^0_i$, as before.
Using Gr\"onwall's lemma on this inequality and on \eqref{eq:dtxixj}, we obtain the result.
\end{proof}

\section{Convergence towards smooth solutions} \label{sec:convergence}

In this section we prove Theorem \ref{th:convergence}, i.e.\ the convergence of discrete solutions towards smooth solutions of the equation
\[
\partial_t \rho - \mathrm{div}\left[ \rho \nabla U'(\rho)\right]= 0\,.
\]
The proof follows similar lines as the one used in \cite{gallouet2022convergence} to analyse the system associated to a different energy regularization, given by \eqref{eq:ftilde}. It relies on the construction of an appropriate relative entropy and on two main technical lemmas. 

\subsection{Preliminary lemmas} 
The first lemma provides us with a way to control the relative pressure
\begin{equation}\label{eq:relpress}
P(r|s) \coloneqq P(r) - P(s)  - P'(s) (r-s)\,
\end{equation}
by $U(r|s)$. In particular, we will make the following assumption: there exists a constant $A>0$ such that
\begin{equation}\label{eq:pbound}
|P''(r)|  \leq A\, U''(r)\, \quad \forall\, r> 0\,.
\end{equation}
This assumption is trivially satisfied for the important case of power energies, i.e.\ when $U(r) = r^\gamma/(\gamma-1)$ with $\gamma>1$, which corresponds to $P(r) = r^\gamma$. It implies the following lemma, which is extracted from Lemma 3.3 in \cite{giesselmann2017relative}.

\begin{lemma} \label{lem:Abound}
Let $U$ and $P$ be smooth functions on $[0,\infty)$  satisfying \eqref{eq:thermodynamic} and \eqref{eq:pbound}. Then
\begin{equation}\label{eq:relpbound}
|P(r|s)| \leq A U(r|s) \, \quad \forall\, r,s>0\,.
\end{equation}
\end{lemma}
\begin{proof}
We have $P(r|s) = (r-s)^2 \int_0^1 (1-\theta) P''((1-\theta)s + \theta r) \,\ed \theta$ and similarly for $U(r|s)$. Hence, using equation \eqref{eq:pbound},
\[
|P(r|s)|\leq (r-s)^2 \int_0^1 (1-\theta) |P''((1-\theta)s + \theta r)| \,\ed \theta \leq A\, U(r|s)\,.
\]%
\end{proof} 

The second lemma is necessary to deal with the fact that the particles may exit the domain $\Omega$, if $\Omega$ is not convex. For this reason, we will need to use an extension of the continuous density $\rho$ on the whole space. It will be clear in the following that such an extension needs to verify the continuity equation with respect to an appropriate velocity field also defined on the whole space. Here we just report a simplified version of the statement of Lemma 4.1 in \cite{gallouet2022convergence}, which was used precisely for this purpose.

\begin{lemma}\label{lem:continuityextension}
Let $u : [0,T]\times \Omega \rightarrow \mathbb{R}^d$ be such that  $u \cdot n_{\partial M} = 0$ on $[0,T]\times \partial \Omega$, and $\rho^0 : \Omega \rightarrow [\rho_{min},\infty)$ with $\rho_{min}>0$. If $u$ is of class $C^{2,1}$ in space, uniformly in time, and $\rho_0$ is of class $C^{1,1}$, then there exist $\tilde{u}: [0,T]\times \mathbb{R}^d \rightarrow \mathbb{R}^d$ and  $\tilde{\rho}: [0,T] \times \mathbb{R}^d \rightarrow \mathbb{R}$ such that:
\begin{enumerate}
\item $\tilde{u}$ is an extension of $u$, i.e.\ $\tilde{u}(t)|_\Omega = u(t)$ for all $t\in [0,T]$, and there exists a constant $C>0$ only depending on $d$ such that
 \begin{equation} \label{eq:boundu1} 
 \sup_{t\in[0,T]}  \| \tilde{u}(t) \|_{C^{2,1}} \leq C \sup_{t\in[0,T]} \| {u}(t) \|_{C^{2,1}}\, ;
 \end{equation}
\item the couple $(\tilde{\rho},\tilde{u})$ solves the continuity equation:
\[
\partial_t \tilde{\rho} + \mathrm{div} (\tilde{\rho} \tilde{u}) = 0 \quad \text{on }\, [0,T] \times \mathbb{R}^d,
\]
and in particular the curve $\rho:t\in[0,T]\rightarrow \tilde{\rho}(t)|_\Omega$ is the unique  solution of the continuity equation on $[0,T] \times \Omega$ associated with $u$ and initial conditions $\rho(0) = \rho^0$; $\tilde{\rho}\geq \tilde{\rho}_{min}>0$, where $\tilde{\rho}_{min}$ only depends on $\rho_{min}$, $\sup_{t\in[0,T]} \|u(t)\|_{C^{2,1}}$, $T$ and $d$; moreover, $\sup_{t\in[0,T]} \|\tilde{\rho}(t)\|_{C^{1,1}}$ only depends on $\|\rho^0\|_{C^{1,1}}$, $\sup_{t\in[0,T]} \|u(t)\|_{C^{2,1}}$, $T$, $d$ and on $\rho_{min}$.
\end{enumerate}
\end{lemma}

\subsection{Main assumptions and relative entropy}
Suppose that $\rho:[0,T]\times \Omega \rightarrow (0,\infty)$ is a sufficiently smooth  solution of equation \eqref{eq:pde0}. In particular, we suppose that $u\coloneqq - \nabla U'(\rho)$ and $\rho(0,\cdot) = \rho^0$ satisfy the assumptions of Lemma \ref{lem:continuityextension}. We will denote by $\tilde{\rho}:[0,T]\times \mathbb{R}^d \rightarrow [0,\infty)$ and $\tilde{u}:[0,T]\times \mathbb{R}^d \rightarrow \mathbb{R}^d$ the extensions of $\rho$ and $u$, respectively, outside the domain. By construction these satisfy the continuity equation on the whole space, but in general outside the domain $\Omega$, 
\[
\tilde{v} \coloneqq -\nabla U'(\tilde{\rho}) \neq \tilde{u}\,.
\]

For any $N>0$, let $X_N:[0,T]\rightarrow (\mathbb{R}^d)^N$ be a solution of the discrete model \eqref{eq:gfdiscrete}, with given initial conditions $x_i(0) =x_i^0 \in \Omega$. We suppose that for a given diffeomorphism $\Phi:\Omega\rightarrow \Omega$ and a smooth reference density $\nu:\Omega \rightarrow (0,\infty)$, the initial density can be written as follows:
\[
\rho^0 = \frac{\nu}{\mathrm{det}(\nabla \Phi)} \circ \Phi^{-1}\,.
\]
Then, given a fixed tessellation $\mc{T}_N = \{T_i\}_i\in \mathbb{T}_N(\Omega)$, we denote
\begin{equation}\label{eq:deltaN}
\delta_N \coloneqq \int_{T_i} |\Phi(x) - x_i^0|^2  \nu(x)\,\ed x\,, \quad m^0_i = \int_{\Phi(T_i)} \rho^0 = \int_{T_i} \nu\,.
\end{equation}
Moreover we require that $C_0^{-1}/ N \leq m^0_i \leq C_0/N$ for a constant $C_0>0$ and for all $1\leq i\leq N$. 

Let us introduce the time-dependent measures $\mu_N:[0,T]\rightarrow \mc{P}(\Omega)$ and $\bar{\mu}_N:[0,T]\rightarrow \mc{P}(\Omega)$ defined as follows
\[
{\mu}_N(t) \coloneqq \sum_{i=1}^N {m_i^0} \delta_{x_i(t)} \,,\quad \bar{\mu}_N(t) \coloneqq \sum_{i=1}^N \frac{m_i^0}{|L_i(t)|} {\bf 1}_{L_i(t)} \ed x\,.
\]
In order to define the relative entropy between the smooth and discrete solutions, we first introduce the flow of $u$, which is the curve of diffeomorphisms $\varphi:[0,T]\times \Omega \rightarrow \Omega$  satisfying
\[
\partial_t \varphi(t,x) = u(t,\varphi(t,x))\,, \quad \varphi(0,x) = \Phi(x)\,.
\]
Then let $\varphi_N(t) \coloneqq (\varphi(t,\Phi^{-1}(x_i^0)))_i \in \Omega^N$ be the collection of the exact trajectories of the particles located at $x_i^0$ at time $t=0$. 
The relative entropy of the discrete solution with respect to the continuous one is defined as follows:
\begin{equation}\label{eq:relentropy}
F_{\varepsilon}(X_N|\rho;t) \coloneqq \sum_i \int_{L_i(t)}\frac{|x-x_i(t)|^2}{2\varepsilon}\,\ed x + \int_\Omega U(\bar{\mu}_N(t)|\rho(t)) + \frac{\|\varphi_N(t)-X_N(t)\|^2_{m^0}}{2}\,,
\end{equation}
where for all $r\geq 0$ and $s>0$,
\[
U(r|s) \coloneqq U(r)- U(s) - U'(s)(r-s).
\]

{
\begin{remark}[$L^2$ error on the flow]\label{rem:l2error} The $L^2$ error on the flow which is present in the estimate \eqref{eq:est} (first term on the left-hand side) can be directly related to the last term of \eqref{eq:relentropy} as follows
\begin{equation}\label{eq:estflow}
\begin{aligned}
\| \varphi^X_N&(t,\cdot) -  \varphi(t,\cdot)\|^2_{L^2(\nu)} =  \sum_i \int_{T_i} |x_i(t) - \varphi(t,y)|^2 \nu(y)\ed y \\
& \leq \sum_i \int_{T_i} 2 ( |x_i(t) - \varphi(t,\Phi^{-1}(x_i^0))|^2  + |\varphi(t,\Phi^{-1}(x_i^0)) - \varphi(t,y)|^2 )\nu(y)\ed y\\
& \leq 2 \| \varphi_N(t) - X_N(t)\|_{m^0}^2 + 2 \mathrm{Lip}(\varphi(t,\Phi^{-1}(\cdot)))^2 \delta_N^2  \,,
\end{aligned}
\end{equation}
where $\delta_N$ is defined in equation \eqref{eq:initialcond0}  and $\mc{T}_N = \{T_i\}_i \in \mathbb{T}_N(\Omega)$ is a fixed reference tessellation as in Section \ref{sec:relationporous}. Since $\varphi(t,\Phi^{-1}(\cdot)))$ is the flow of $-\nabla U'(\rho)$ with the identity as initial condition, one can control $\mathrm{Lip}(\varphi(t,\Phi^{-1}(\cdot)))$ by a constant only depending on $\|\nabla U'(\rho)\|_{C^{1,1}}$ and $T$. Given this, the estimate that we will derive in the following for \eqref{eq:relentropy} implies directly \eqref{eq:est} for an appropriate constant $C>0$ as in the statement of the theorem. 
\end{remark}
}

\subsection{Time derivative of the relative entropy} We now compute the time derivative of the relative entropy and isolate the terms that need to be estimated.
It will be useful to define the following quantity:
\[
H(t) \coloneqq \int_{\mathbb{R}^d} U'(\tilde{\rho}(t)) \ed(\mu_N(t) -\bar{\mu}_N(t))\,,
\]
where $\bar{\mu}_N$ is extended by zero on $\mathbb{R}^d$. We will keep using this convention in what follows.

Let us start by rewriting the relative entropy as follows:
\begin{multline}\label{eq:estimate1}
F_{\varepsilon}(X_N|\rho;t) = F_\varepsilon(X_N(t)) +  \int_\Omega P(\rho(t)) -  \int_{\mathbb{R}^d} U'(\tilde{\rho}(t)) \ed \mu_N(t) \\ +  H(t) + \frac{\|\varphi_N(t)-X_N(t)\|^2_{m^0}}{2}\,.
\end{multline}
We compute the time derivative of the terms of the right-hand side separately.
For the first term, we write
\[
\begin{aligned}
\frac{\ed}{\ed t} F_\varepsilon(X_N(t)) = & \sum_i \int_{L_i(t)} \frac{x_i(t) - x}{\varepsilon}\cdot \dot{x}_i(t) \ed x \\
 = &  \, -\langle \dot{X}_N(t), \dot{X}_N(t) - \tilde{u}(t,X_N(t)) \rangle_{m^0} \\&  + \sum_i \int_{L_i(t)} \frac{x_i(t)-x}{\varepsilon} \cdot (\tilde{u}(t,x_i(t)) -\tilde{u}(t,x)) \ed x \\&   +\sum_i \int_{L_i(t)} \frac{x_i(t)-x}{\varepsilon} \cdot u(t,x) \,\ed x\,,
\end{aligned}
\]
and by Lemma \ref{lem:optimalityu} we can write the last term on the right-hand side as follows:
\[
\sum_i \int_{L_i(t)} \frac{x_i(t)-x}{\varepsilon} \cdot u(t,x) \ed x =  \sum_i \int_{L_i(t)} \frac{|x-x_i(t)|^2}{2\varepsilon} \ed x - \int_\Omega P (\bar{\mu}_N(t)) \mathrm{div}u(t)\,.
\]

For the second term, using the continuity equation $-\partial_t{\rho} =  \nabla {\rho} \cdot u + {\rho} \mathrm{div} u$, we obtain
\[
\begin{aligned}
 \frac{\ed}{\ed t} \int_{\Omega} P(\rho(t)) & =  -\int_{\Omega}  P'(\rho(t))\rho(t) \mathrm{div} u(t) - \int \nabla P(\rho(t)) \cdot u(t) \\
 & = \int_{\Omega} \left[ P(\rho(t)) - P'(\rho(t))\rho(t) \right] \mathrm{div} u(t) \,.
 \end{aligned}
\]
Finally, for the third term, using again the continuity equation $-\partial_t\tilde{\rho} =  \nabla \tilde{\rho} \cdot \tilde{u} + \tilde{\rho} \mathrm{div} \tilde{u}$, this time on $\mathbb{R}^d$, we have
\[
\begin{aligned}
\frac{\ed}{\ed t} \int_{\mathbb{R}^d} U'(\tilde{\rho}(t)) \ed \mu_N(t)& = \sum_i \nabla U'(\tilde{\rho}(t,x_i)) \cdot \dot{x}_i m^0_i  + \int_{\mathbb{R}^d} U''(\tilde{\rho}(t)) \partial_t \tilde{\rho} \ed \mu_N(t)\\
& = \sum_i \nabla U'(\tilde{\rho}(t,x_i)) \cdot (\dot{x}_i - \tilde{u}(t,x_i)) m^0_i  - \int_{\mathbb{R}^d} P'(\tilde{\rho}(t))  \mathrm{div}{\tilde{u}(t)} \ed \mu_N(t)\\
& = -\langle \tilde{v}(t,X_N(t)),  \dot{X}_N(t) - \tilde{u}(t,X_N(t))\rangle_{m^0}  - \int_{\mathbb{R}^d} P'(\tilde{\rho}(t))  \mathrm{div}{\tilde{u}(t)} \ed \mu_N(t)\,.
\end{aligned}
\]
Reinserting these expressions into the time derivative of \eqref{eq:estimate1} and rearranging terms we obtain
\begin{equation}\label{eq:estimate}
\frac{\ed}{\ed t} F_{\varepsilon}(X_N|\rho;t) + \frac{\ed}{\ed t} H(t) + \| \dot{X}_N(t) - \tilde{u}(t,X_N(t)) \|^2_{m^0} =   \sum_{j=1}^5 I_j(t) \,,
\end{equation}
where the terms in the sum on the right-hand side are defined as follows:
\[
I_1(t)\coloneqq \sum_i \int_{L_i(t)} \frac{x_i(t)-x}{\varepsilon} \cdot (\tilde{u}(t,x_i(t)) -\tilde{u}(t,x)) \ed x\,,
\]
\[
I_2(t) \coloneqq \sum_i \int_{L_i(t)}\frac{|x-x_i(t)|^2}{2\varepsilon} \mathrm{div}u(t,x) \ed x -\int_\Omega P(\bar{\mu}_N(t)|\rho(t)) \mathrm{div} u(t) \,, \]
\[
I_3(t) \coloneqq \int_{\mathbb{R}^d} P'(\tilde{\rho}(t))  \mathrm{div}{\tilde{u}(t)} \ed (\mu_N(t)-\bar{\mu}_N(t))\,,
\]
\[
I_4(t) \coloneqq \langle \tilde{v}(t,X_N(t)) - \tilde{u}(t,X_N(t)),  \dot{X}_N(t) - \tilde{u}(t,X_N(t))\rangle_{m^0} \,,
\]
\[
I_5(t) \coloneqq \langle u(t,\varphi_N(t)) - \dot{X}_N, \varphi_N(t) - X_N(t) \rangle_{m^0}\,.
\]
\subsection{Uniform estimates} In the following, for any given Lipchitz function $f\in {C}^{0,1}(\Omega)$, we will denote by $\mathrm{Lip}(f)$ its Lipschitz constant, and for any time-dependent function $g \in C([0,T]; C^{0,1}(\Omega))$ we denote by $\mathrm{Lip}_{T} g \coloneqq \sup_{t\in[0,T]} \mathrm{Lip}(g(t,\cdot))$, and similarly for vector-valued functions. 

We estimate separately the terms on the right-hand side of \eqref{eq:estimate}. We have
\[
I_1(t) \leq \mathrm{Lip}_T(\tilde{u}) \sum_i \int_{L_i(t)} \frac{|x_i(t)-x|^2}{\varepsilon} \ed x\,,
\]
and using Lemma \ref{lem:Abound},
\[
I_2(t) \leq  (\mathrm{Lip}_T(u)+A) \sum_i \int_{L_i(t)}\frac{|x-x_i(t)|^2}{2\varepsilon}\ed x +\int_\Omega U(\bar{\mu}_N(t)|\rho(t)) \mathrm{div} u(t) \,. \]
To bound $I_3$, let us introduce $h\coloneqq P(\tilde{\rho}) \mathrm{div}(\tilde{u})$. Then, for any $\lambda>0$,
\[
\begin{aligned}
I_3(t)& = \sum_i \left[ h(t,x_i(t))m^0_i - \frac{m^0_i}{|L_i(t)|} \int_{L_i(t)} h(t,x) \ed x\right] \\
& \leq \sum_i \mathrm{Lip}_T(h) \frac{m^0_i}{|L_i(t)|}\int_{L_i(t)} |x_i(t) -x| \ed x \\
& \leq \mathrm{Lip}_T(h)  \left[ \sum_i \int_{L_i(t)} \frac{|x_i(t) -x|^q}{\lambda^q q\varepsilon} \ed x + \frac{\varepsilon^{p-1}\lambda^p}{p} \sum_i \left(\frac{m^0_i}{|L_i(t)|}\right)^p |L_i(t)| \right]\,,
\end{aligned}
\]
where $p,q>1$ are conjugate exponents, i.e.\ $1/p + 1/q=1$. Recall that we supposed that there exist $R,\alpha>1$ and $\beta>0$, such that 
\[
U(r) - \inf U \geq \beta r^\alpha \quad \forall\, r\geq R\,.
\]
Then, choosing $p = \min\{2, \alpha\}$ we get $q\geq 2$, and therefore
\begin{multline*}
I_3(t)\leq  \mathrm{Lip}_T(h)   \frac{2\mathrm{diam}(\Omega)^{q-2}}{q\lambda^q} \sum_i \int_{L_i(t)} \frac{|x_i(t) -x|^2}{2\varepsilon} \ed x \\  +\mathrm{Lip}_T(h) \frac{\varepsilon^{p-1}\lambda^p}{p} \left( |\Omega|R^p + \beta ^{-1} \int_\Omega U(\bar{\mu}_N(t)) - \beta^{-1} |\Omega|\inf U  \right) \,,
\end{multline*}
where we used the fact that, since $X_N(0) \in \Omega^N\setminus \Delta_N$, by Lemma \ref{lem:hull}  $x_i(t) \in \mathrm{conv}(\Omega)$ (the convex hull of $\Omega$) for all times $t\geq 0$. Hence,
\begin{equation}
I_3(t)\leq\mathrm{Lip}_T(h) \left(  \frac{C_1}{\lambda^q} F_\varepsilon(X|\rho;t) + C_2 \lambda^p {\varepsilon^{p-1}}\right)\,,
\end{equation}
where 
\begin{equation}\label{eq:constants}
C_1\coloneqq     \frac{2\mathrm{diam}(\Omega)^{q-2}}{q} \,, \quad C_2 \coloneqq \frac{\beta^{-1} F_\varepsilon(X(0))+|\Omega|(R^p - \beta^{-1}\inf U) }{p}.
\end{equation}
Using the same arguments to bound $H(t)$, we get
\begin{equation}\label{eq:Hbound}
|H(t)| \leq \frac{1}{2} F_\varepsilon(X|\rho;t) + C_2 \bar{\lambda}^p {\varepsilon^{p-1}}\,,
\end{equation}
where \[ \bar{\lambda} = \left(\frac{2}{C_1 \max\{\mathrm{Lip}_T(h),1\}}\right)^{\frac{1}{q}}.\]

Finally we observe that
\[
\begin{aligned}
I_4(t) +I_5(t) & = \langle \tilde{v}(t,X_N(t)) - \tilde{v}(t,\varphi_N(t)),  \dot{X}_N(t) - \tilde{u}(t,X_N(t))\rangle_{m^0} \\
&\quad + \langle \tilde{u}(t, \varphi_N(t)) - \tilde{u}(t,X_N(t)),  \dot{X}_N(t) - \tilde{u}(t,X_N(t))\rangle_{m^0} \\
&\quad + \langle \varphi_N(t) - X_N(t),  \dot{X}_N(t) - \tilde{u}(t,X_N(t))\rangle_{m^0} \\
& \leq \left(\mathrm{Lip}(\tilde{v})^2 + \mathrm{Lip}(\tilde{u})^2 +1 \right) \|X_N(t)- \varphi_N(t)\|^2_{m^0} + \frac{\| \dot{X}_N(t) - \tilde{u}(t,X_N(t))\|^2_{m^0}}{2}\,.
\end{aligned}
\]

\subsection{Gr\"onwall argument}
Reinserting the estimates above into \eqref{eq:estimate}, we obtain
\begin{multline}\label{eq:estimatetot}
\frac{\ed}{\ed t} F_{\varepsilon}(X_N|\rho;t) + \frac{\ed}{\ed t} H(t) + \frac{\| \dot{X}_N(t) - \tilde{u}(t,X_N(t)) \|^2_{m^0}}{2}\\ \leq C_3 F_{\varepsilon}(X_N|\rho;t) +C_4\varepsilon^{p-1}\,.
\end{multline}
Let $G(t)\coloneqq F_{\varepsilon}(X_N|\rho;t) + H(t)$ and observe that equation \eqref{eq:Hbound} implies
\[
- F_{\varepsilon}(X_N|\rho;t) \leq 2 H(t) + 2 C_2 \bar{\lambda}^p {\varepsilon^{p-1}}\,,
\]
and adding  $2F_{\varepsilon}(X_N|\rho;t)$ on both sides we obtain
\[
 F_{\varepsilon}(X_N|\rho;t) \leq 2 G(t) + 2 C_2 \bar{\lambda}^p {\varepsilon^{p-1}}\,.
\]
Substituting this into \eqref{eq:estimatetot}, we obtain
\[
\frac{\ed}{\ed t} G(t) + \frac{\| \dot{X}_N(t) - \tilde{u}(t,X_N(t)) \|^2_{m^0}}{2}\\ \leq 2 C_3 G(t) +C_5\varepsilon^{p-1}\,.
\]
Hence by Gr\"onwall's inequality we get
\begin{multline}
F_{\varepsilon}(X_N|\rho;t) + \int_0^t\frac{\| \dot{X}_N(s) - \tilde{u}(t,X_N(s)) \|^2_{m^0}}{2}\ed s \\\leq \exp(2C_3 t) G(0) + \frac{C_5}{2C_3}\varepsilon^{p-1}(\exp(2C_3 t)-1) - H(t)\,,
\end{multline}
and using again the bound on $H(t)$ in \eqref{eq:Hbound} we find
\[
\frac{F_{\varepsilon}(X_N|\rho;t)}{2}  + \int_0^t\frac{\| \dot{X}_N(s) - \tilde{u}(t,X_N(s)) \|^2_{m^0}}{2}\ed s \\\leq \exp(2C_3 t) G(0) + C_6 \varepsilon^{p-1}\,. 
\]
\subsection{Estimates on the initial datum} In order to conclude we only need to estimate $G(0)$ and the initial energy $F_\varepsilon(0)$, since $C_6$ is an affine function of the latter, due to \eqref{eq:constants}. Recall the definition of $\delta_N$ in \eqref{eq:deltaN}.
Using Jensen's inequality and the expression for $\rho^0$ in \eqref{eq:rho0push}, we obtain
\begin{equation}\label{eq:estimateF0}
\begin{aligned}
F_\varepsilon(X_N(0)) & \leq \sum_i \int_{\Phi(T_i)} \frac{|x-x_i^0|^2}{2\varepsilon}\ed x + \sum_i U\left( \frac{m^0_i}{|\Phi(T_i)|} \right) |\Phi(T_i)|\\
& \leq \sum_i \int_{T_i} \frac{|\Phi(x)-x_i^0|^2}{2\varepsilon}\frac{\nu(x)}{(\rho_0\circ \Phi)(x)} \ed x + \int_\Omega U(\rho^0)\\
& = C_7 \frac{\delta_N^2}{2\varepsilon} + \int_\Omega U (\rho^0)\,,
\end{aligned}
\end{equation}
where $C_7= \rho_{min}^{-1}$. Moreover,
\begin{equation}\label{eq:estimatelin0}
\begin{aligned}
\left|\int_{\mathbb{R}^d} U'(\rho^0) \ed (\mu_N(0)-\rho^0)\right| & =\left| \sum_i \int_{\Phi(T_i)} (U'(\rho^0(x_i)) -U'(\rho^0)) \rho^0  \right| \\ 
& \leq \mathrm{Lip}(U'(\rho^0)) \sum_i \int_{T_i}|x_i^0- \Phi(x)| \nu(x) \ed x\\
&\leq \frac{\mathrm{Lip}(U'(\rho^0))^2}{2} \rho^0[\Omega] \varepsilon + \frac{\delta_N^2}{2\varepsilon}\,,
\end{aligned}
\end{equation}
where $\rho^0[\Omega]$ is the integral of $\rho^0$ over $\Omega$.
Hence, combining \eqref{eq:estimateF0} and \eqref{eq:estimatelin0} we obtain
\[
\begin{aligned}
G(0) & = F_\varepsilon(X_N(0)) - \int_\Omega U (\rho^0) - \int_{\mathbb{R}^d} U'(\rho^0) \ed (\mu_N(0)-\rho_0)\\
& \leq (C_7+1) \frac{\delta_N^2}{2\varepsilon}+\frac{\mathrm{Lip}(U'(\rho^0))^2}{2} \rho^0[\Omega] \varepsilon\,.
\end{aligned}
\]
Combining the estimates above we finally find
\begin{multline}\label{eq:finalestimate}
\frac{F_{\varepsilon}(X_N|\rho;t)}{2}  + \int_0^t\frac{\| \dot{X}_N(s) - \tilde{u}(t,X_N(s)) \|^2_{m^0}}{2}\ed s \\\leq \exp(2C_3 t) \left((C_7+1)\frac{\delta_N^2}{2\varepsilon}+\frac{\mathrm{Lip}(U'(\rho^0))^2}{2}\rho^0[\Omega]\varepsilon\right) + C_6 \varepsilon^{p-1}\,, 
\end{multline}
where $C_6$ is an affine function of $\delta_N^2/\varepsilon$, which concludes the proof of Theorem \ref{th:convergence}.

\subsection{External potentials}\label{sec:external} We now consider a slight modification of the original system where the total energy is given by
\[
E_\varepsilon(X) = F_\varepsilon(X) + \sum_i V(x_i)m^0_i\,,
\]
where $V:\mathbb{R}^d\rightarrow \mathbb{R}$ is a Lipschitz function. The gradient flow of this energy, i.e.\ the trajectories satisfying $\dot{X} = -\nabla_{m^0} E_\varepsilon(X)$ solve the following modified system ODEs
\begin{equation}\label{eq:odemod}
\dot{x}_i(t) = -\frac{|L_i(t)|}{m_i^0} \frac{x_i(t) -b_i(t)}{\varepsilon} - \nabla V(x_i(t))
\end{equation}
In this case the limit PDE is given by
\begin{equation}\label{eq:pdemod}
\left\{
\def\arraystretch{1.5}
\begin{array}{ll}
\partial_t \rho  - \mathrm{div}\left[ \rho (\nabla U'(\rho) + \nabla V)\right] = 0 & \text{ on }(0,T) \times \Omega\,,\\
(\nabla U'(\rho) + \nabla V) \cdot n_{\partial \Omega} = 0 &\text{ on }(0,T) \times \partial \Omega\,.
\end{array}\right.
\end{equation}
The proof above also apply to this case with some minor changes. First of all, we observe that the velocity field is now ${u} = -\nabla U'(\rho) -\nabla V$. We assume that this is sufficiently smooth so that the extension Lemma \ref{lem:continuityextension} applies. Then, using the same modulated energy as above, the only different term in equation \eqref{eq:estimate} is $I_4(t)$ which should be replaced by
\[
\tilde{I}_4(t) \coloneqq \langle \tilde{v}(t,X_N(t))-\nabla V(X_N(t)) - \tilde{u}(t,X_N(t)) ,  \dot{X}_N(t) - \tilde{u}(t,X_N(t))\rangle_{m^0}\,,
\]
where as before $\tilde{v} = - \nabla U'(\tilde{\rho})$. This can be controlled exactly as above, leading to the same convergence result as in Theorem \ref{th:convergence}, but with $\nabla U'(\rho)$ replaced by $\nabla U'(\rho) + \nabla V$.

\section{Time discretization and numerical tests}
\subsection{Time discretization}\label{sec:timedisc} In order to compute numerically the  solution of the discrete model \eqref{eq:gfdiscrete} on a given time interval $[t_0,T]$, we will consider the same explicit time discretization used in \cite{gallouet2022convergence} and originally proposed by Brenier in \cite{brenier2000derivation}. Given a time step $\tau=|T-t_0|/N_T>0$ with $N_T\in \mathbb{N}$, define the discrete solution $(X^n)_{n=0}^{N_T}$ as follows: given $X^0$, compute $X^{n+1} = (x_i^{n+1})_{i=1}^N$ for $n\geq 0$ by
\begin{equation}\label{eq:timediscretegf}
x^{n+1}_i =  b^n_i +  \exp\left( -\frac{|L_i^n|}{m^0_i \varepsilon}\tau\right) (x^{n}_i - b^{n}_i)\,,
\end{equation}
where $L^n_i$ and $b^n_i$ are the $i$th cell of the optimal tessellation at the $n$th step and its barycenter, respectively. This scheme can be obtained by following on each time interval $[n\tau, (n+1)\tau]$ the gradient flow of the energy
\[
 \tilde{F}^n_\varepsilon(X) = \sum_{i} \int_{L_i^n} \frac{|x-x_i|^2}{2\varepsilon}\,\ed x + \sum_i U\left(\frac{m^0_i}{|L_i^n|}\right)|L^n_i|\,,
\]
where $L_i^n$ is fixed. In other words, we solve exactly on $[n\tau, (n+1)\tau]$ the (decoupled) system of ODEs
	\[
	\dot{x}_i(t) = - \frac{1}{m^0_i} \partial_{x_i} \tilde{F}^n_\varepsilon(X) = -\frac{|L_i^n|}{m_i^0} \frac{x_i(t) -b_i^n}{\varepsilon}\,,
	\]
with initial conditions $x_i(n\tau)=x_{i}^n$, for $i=1,\ldots,N$. Then we set $x_i^{n+1}=x_i((n+1)\tau)$, which is equivalent to equation \eqref{eq:timediscretegf}, and finally we compute the optimal tessellation $(L_i^{n+1})_i$ and barycenters $(b_i^{n+1})_i$ associated with the new particle positions.
As a consequence of the definition of the discrete energy \eqref{eq:FepsX}, $F_\varepsilon$ is dissipated by the discrete process defined by \eqref{eq:timediscretegf}:
\[
 {F}_\varepsilon(X^{n+1}) \leq  \tilde{F}^n_\varepsilon(X^{n+1}) \leq \tilde{F}^n_\varepsilon(X^n) =  {F}_\varepsilon(X^{n})\,.
\]
\subsection{Numerical tests}
In this section we present some numerical tests to verify the convergence estimates of Section \ref{sec:convergence}. All the experiments correspond to the case where $\mathbb{L}_N(\Omega) = \mathbb{T}_N^s(\Omega)$. 
The computation of the energy and optimal tessellation is perfermed using Newton's method applied to the system of optimality conditions for the vector of weights $w\in \mathbb{R}^N$ given in \eqref{eq:opticondw}, similarly to the case of semi-discrete optimal transport described in \cite{kitagawa2019convergence}. Computationally, this is simpler than the case of the Moreau-Yosida regularisation \eqref{eq:ftilde} considered in \cite{gallouet2022convergence,leclerc2020lagrangian}, as the optimality conditions in \eqref{eq:opticondw} do not require computing integrals of nonlinear functions over the cells. The scheme was implemented using the open-source library \texttt{sd-ot}, which is available at \url{https://github.com/sd-ot}.

\subsubsection{Barenblatt test case} We consider the case where $U(r) = r^\gamma/(\gamma-1)$ and $P(r) = r^\gamma$ with $\gamma>1$, in which case the corresponding PDE \eqref{eq:pde0} is the porous medium equation. For this energy, we have an exact solution on $\mathbb{R}^d$ which is given by the Barenblatt profile:
\begin{equation}\label{eq:barenblatt}
\rho(t,x) = \frac{1}{{t}^{\alpha}} \left (C^2 - \frac{k}{{t}^{2\beta}} |x|^2 \right)^{\frac{1}{\gamma-1}}_+\,,
\end{equation}
where 
\[
\alpha = \frac{d}{d(\gamma-1)+2} \,, \quad \beta =\frac{\alpha}{d}\,, \quad k = \frac{\beta(\gamma-1)}{2\gamma}\,.
\]
The exact flow is given by
\[
\varphi(t,x) = \left( \frac{t}{t_0} \right)^\beta x\,.
\]
Note that this case falls outside the hypotheses of our theorem, due to lack of a positive lower bound on the density. Note also that since the solution has a compact support the choice of the domain $\Omega$, if sufficiently large, has no impact on the results. 

{
We solve the discrete system on the interval $[t_0,T]$ with $t_0=1/16$, $T=1$, and $C=1/2$, and using $\varepsilon = 10/N$ and $\tau = 10/N^2$. The initial conditions for the particle model are defined via equations \eqref{eq:rho0push}, \eqref{eq:initialcond0} and \eqref{eq:x0}, where $\Phi$ is a radial map from a reference ball of given radius (on which we set $\nu =1$) to the support of $\rho_0$, which can be computed explicitly from \eqref{eq:barenblatt}, and $\mc{T}_N$ is a Voronoi tessellation of the reference ball with $h_N \propto 1/\sqrt{N}$. For all tests we will monitor the weighted $l^2$ error of the flow at the final time $T$, defined as  
\begin{equation}\label{eq:deltaphi}
\Delta \varphi \coloneqq \frac{1}{M} \| X_N(T) - \varphi_N(T) \|_{m^0}\,,
\end{equation}
where $\varphi_N(t) = \varphi(t,X_N(0))$, $\varphi$ is the exact flow associated with the vector field $-\nabla U'(\rho)$, and where $M$ is the total mass.  We stress that up to a rescaling this is precisely the error in the flow which we introduced in the definition of the relative entropy \eqref{eq:relentropy}, and it can be regarded as an $\mc{O}(\delta_N)$ approximation of the standard $L^2$ error on the  Lagrangian flow as discussed in Remark \ref{rem:l2error}. The results in Figure \ref{fig:convergence} show generally a faster convergence than that predicted by Theorem \ref{th:convergence} but confirm a dependence of the convergence rates on the growth rate of the internal energy function $U$.}

\begin{figure}
\includegraphics[scale =.77,trim = 13 0 0 0, clip]{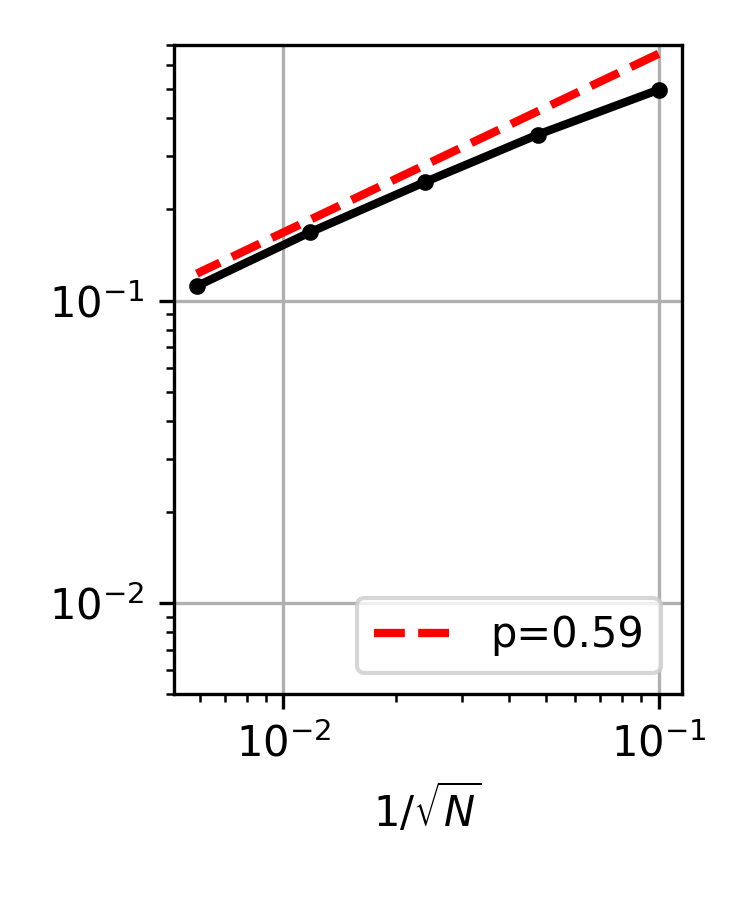}
\includegraphics[scale =.77,trim = 13 0 0 0, clip]{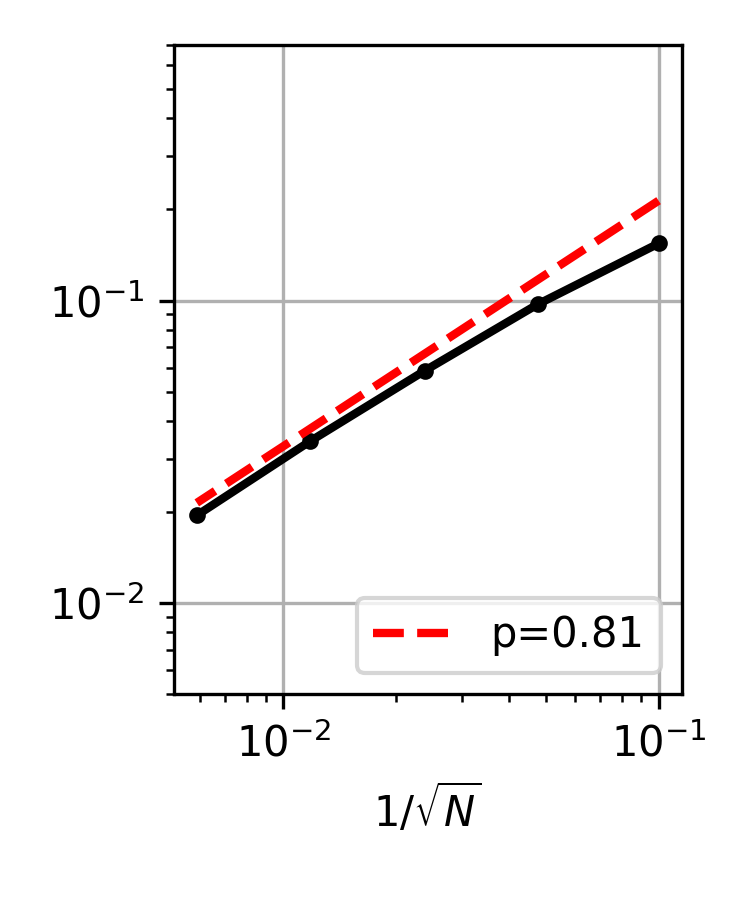}
\includegraphics[scale =.77,trim = 13 0 0 0, clip]{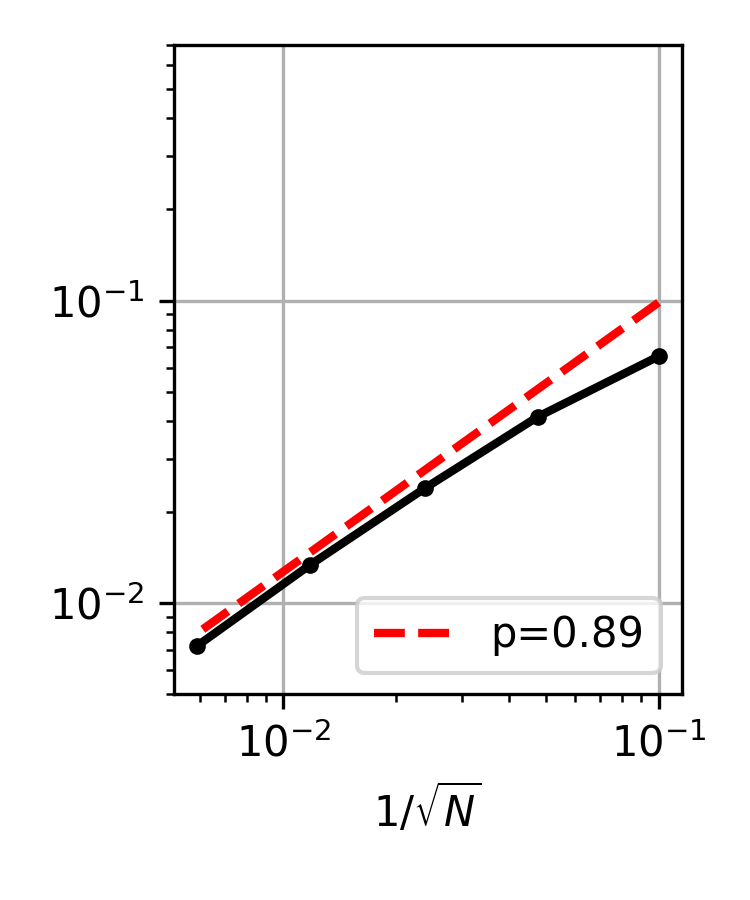}
\caption{{Error $\Delta \varphi$ defined in equation \eqref{eq:deltaphi} as a function of $1/\sqrt{N}$, for $\gamma = 1.5$ (left), $\gamma =2$ (center) and $\gamma=4$ (right).
The curves are compared to $h_N^p \sim N^{-p/2}$, corresponding to $p$th order of convergence, with $p$ evaluated over the last refinement step.
}\label{fig:convergence}}
\end{figure}

\subsubsection{Quadratic potential} We consider again the internal energy function $U(r)=r^2$, but with an additional quadratic potential $V(x) = |x - \bar{x}|^2/2$, driving the particles towards $\bar{x}\in \mathbb{R}^d$. As described in Section \ref{sec:external}, the discrete model is now defined by the system of ODEs \eqref{eq:odemod}, and we can apply the same time discretization strategy described in Section \ref{sec:timedisc}, which leads to the scheme
\[
x^{n+1}_i =  c^n_i +  \exp\left( -\lambda^n_i \tau\right) (x^{n}_i - c^{n}_i)\,,
\]
where
\[
\lambda^n_i = \frac{|L^n_i|}{m^0_i\varepsilon} +1 \quad \text{and}\quad
c^n_i = b^n_i + \frac{ \bar{x} - b^n_i}{\lambda^n_i}\,.
\]

In this case the density in the continuous model \eqref{eq:pdemod} converges exponentially towards the Barenblatt profile $\rho_\infty(x) = \max( (\frac{M}{2\pi})^{1/2} - \frac{1}{4}|x-\bar{x}|^2,0)$ where $M$ is the total mass. Here, we consider as initial condition a configuration where the particles are equally spaced within a cross of unit height and width, with barycenter at $\bar{x}=0$, and share the same mass, $m^0_i = M/N$ for all $i$. In particular we set $M=0.12$, $N =1.23 \cdot 10^4$, $\tau = 1/3\cdot 10^{-2}$, $\varepsilon = 2/3 \cdot 10^{-2}$. Figures \ref{fig:cross} and \ref{fig:crossenergy} show the particle distribution at different times and the energy evolution, respectively, and show the exponential decay of the density towards the equilibrium distribution.  Note, in particular, that at the steady state the size of the data attachement term in the discrete energy (i.e., the first term in $F_\varepsilon$) has roughly the same size as the difference between the exact energy and the one computed using the density reconstruction \eqref{eq:muNfun} (i.e., the second term in $F_\varepsilon$).

\begin{figure}
\includegraphics[scale =.65,trim = 15 0 50 0, clip]{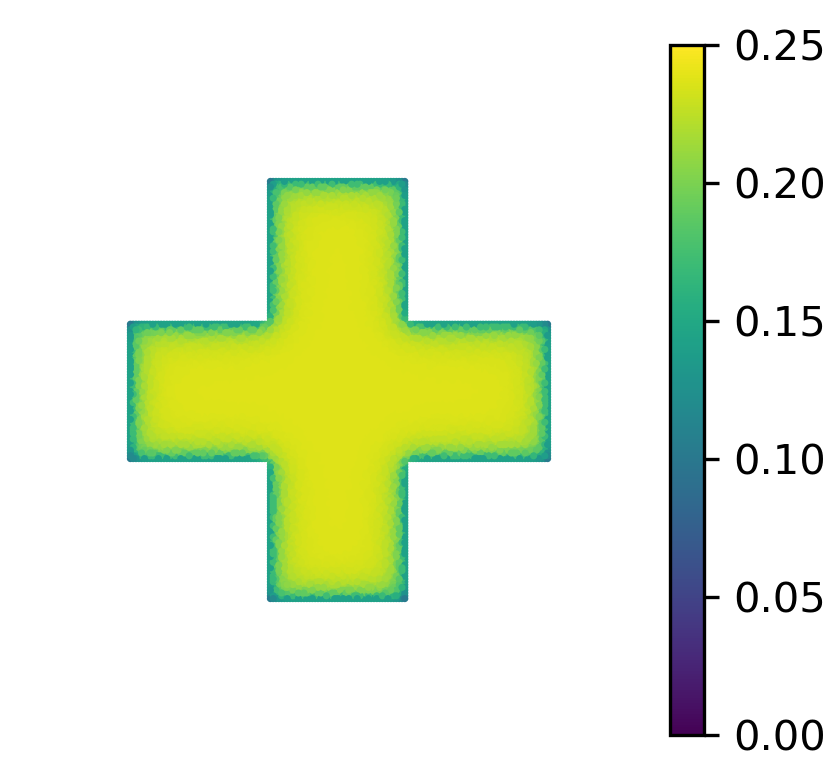}
\includegraphics[scale =.65,trim = 15 0 50 0, clip]{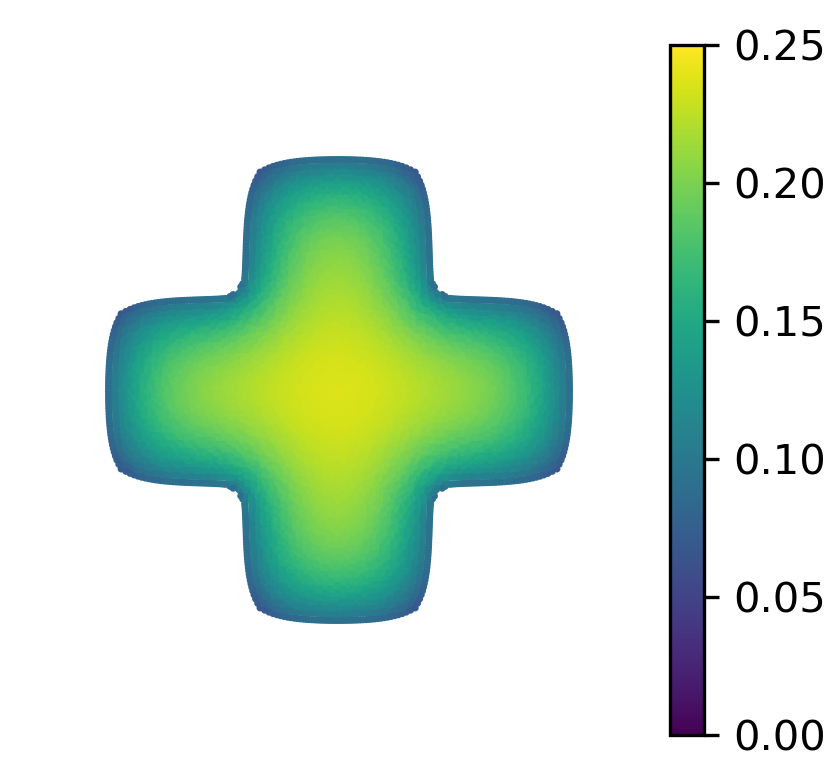}
\includegraphics[scale =.65,trim = 15 0 50 0, clip]{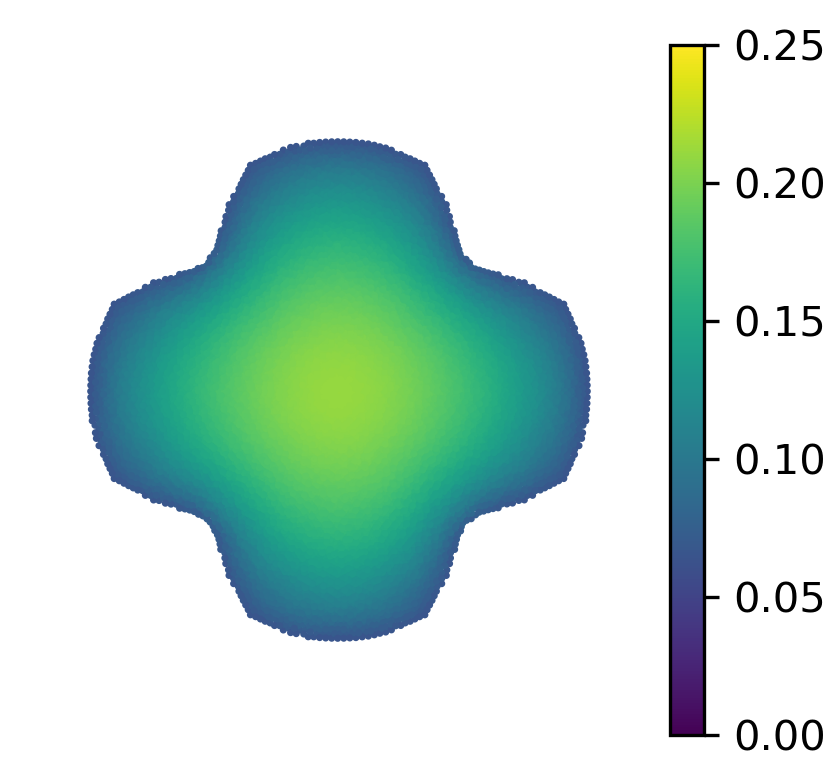}
\includegraphics[scale =.65,trim = 15 0 0 0, clip]{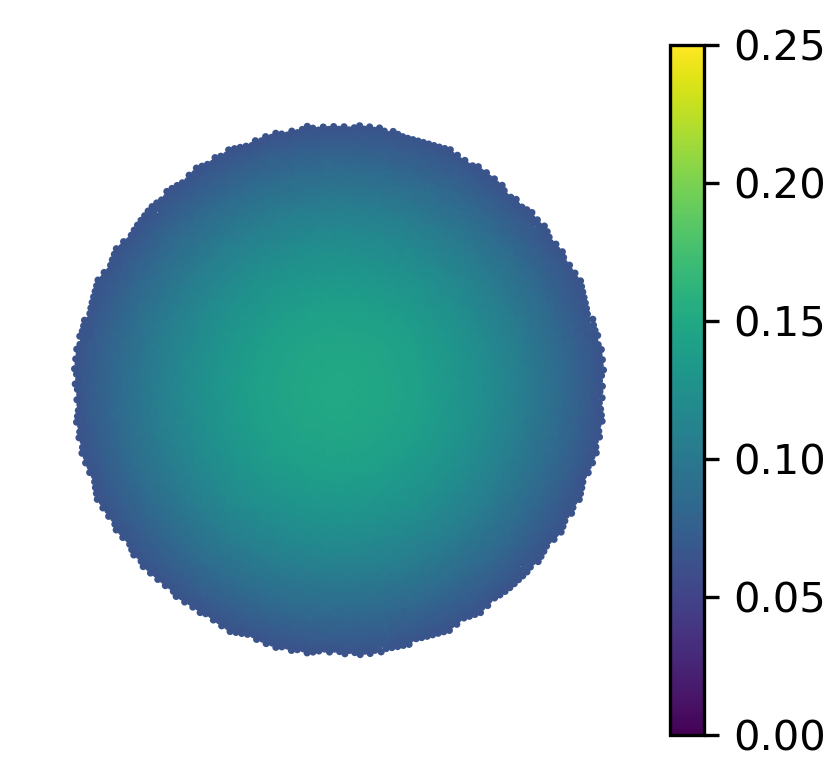}
\caption{Scatter plot of the particle positions at different times (from left to right, $t=0, 0.05, 0.2, 8$) for the quadratic potential test case. The color scale refers to the density, computed for each particle as $m^0_i/|L_i|$.}\label{fig:cross}
\end{figure}

\begin{figure}
\includegraphics[scale=.8,trim = 0 10 0 10, clip]{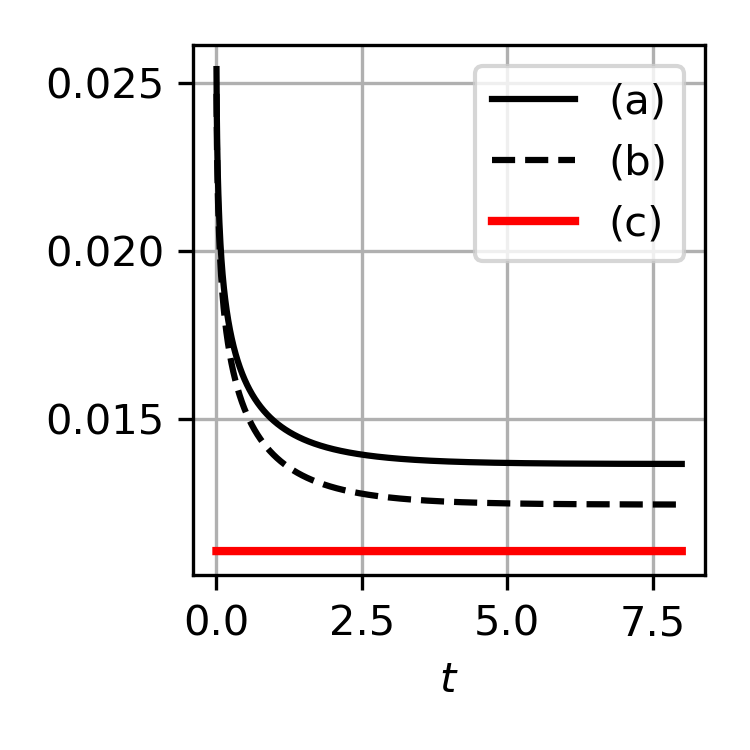}
\caption{Energy evolution for the quadratic potential test case: (a) $F_\varepsilon(X)$; (b) $\sum_i U(m^0_i/|L_i|)|L_i|$; (c) internal energy of the equilibrium density $\int U(\rho_\infty)$.}\label{fig:crossenergy}
\end{figure}

\section*{Acknowledgements}
This work was partly supported by the Labex CEMPI (ANR-11-LABX-0007-01).

\appendix
\section{Lagrangian formulation of porous media and link with the discrete model}
\label{sec:applag}
In this section we describe formally the gradient flow structure of the porous medium equation \eqref{eq:pde0} in Lagrangian variables \cite{evans2005diffeomorphisms}. At the Eulerian level, this corresponds to the Wasserstein gradient flow formulation originally put forward by Otto \cite{otto2001geometry}.

Let us denote by $\mc{M}_+(\Omega)$ the set of positive measures on $\Omega$.  Given a measurable map $\Psi: \Omega \rightarrow \mathbb{R}^d$, and a measure $\rho \in \mc{M}_+(\Omega)$ the pushforward of $\rho$ by $\Psi$ is the measure $\Psi_\# \rho \in \mc{M}_+(\mathbb{R}^d)$, satisfying
\[
\int_{\mathbb{R}^d} f \ed \Psi_\# \rho = \int_\Omega f(\Psi(x)) \ed \rho (x)
\]
for all $f \in C_0(\mathbb{R}^d)$, the space of continuous functions vanishing at infinity.

Consider now two reference measures $\mu,\nu \in \mc{M}_+(\Omega)$ with smooth and strictly positive densities with respect to the Lebesgue measure on $\Omega$, denoted $\mathrm{Leb}$ in the following, and let us define the  energy $\mc{F}:L^2_{\nu}(\Omega;\mathbb{R}^d) \rightarrow \mathbb{R}$ by
\[
\mc{F}(\varphi) = \left\{
\begin{array}{ll} \displaystyle
\int_\Omega U \left(\frac{\ed \varphi_\#\nu}{\ed \mu}\right) \ed \mu & \text{ if }  \varphi_\#\nu \ll \mu\,, \\
+ \infty   &\text{otherwise}\,.
\end{array}
\right.
\]
The gradient of this energy with respect to the $L^2_{\nu}$ metric, at a given configuration $\varphi \in\mathrm{Diff}(\Omega)$, can be defined as follows. Consider a smooth curve $ (-\varepsilon,\varepsilon)\ni s  \rightarrow \varphi(s) \in \mathrm{Diff}(\Omega) \subset L^2_{\nu}(\Omega;\mathbb{R}^d)$ such that $\varphi(0) = \varphi$ and $\ed \varphi(s) /\ed s |_{s=0} = \delta \varphi $. Then,
\[
\begin{aligned}
\langle \nabla_{L^2_{\nu}} \mc{F}(\varphi), \delta \varphi \rangle& =  \frac{\ed}{\ed s}\Big|_{s=0} \mc{F}(\varphi(s)) \\ &=
\int U' \left( \frac{\ed \varphi_\# \nu}{\ed \mu} \right) \frac{\ed}{\ed s}\Big|_{s=0} \frac{\ed \varphi(s)_\# \nu}{\ed \mu} \ed \mu \\
&=
\frac{\ed}{\ed s}\Big|_{s=0} \int U' \left( \frac{\ed \varphi_\# \nu}{\ed \mu} \right) \circ \varphi(s) \ed \nu \\
&= 
\Big\langle \nabla U'\left( \frac{\ed \varphi_\# \nu}{\ed \mu} \right) \circ \varphi , \delta \varphi \Big\rangle\,,
\end{aligned}
\]%
where $\langle \cdot,\cdot\rangle$ denotes the inner product on $L^2_{\nu}(\Omega;\mathbb{R}^d)$. Then, we can interpret any smooth curve of diffeomorphisms $\varphi:[0,T] \rightarrow \mathrm{Diff}(\Omega)$ satisfying
\[
\left\{
\begin{array}{l}\displaystyle
\partial_t \varphi(t) = - \nabla U'\left( \frac{\ed [\varphi(t)_\# \nu]}{\ed \mu} \right) \circ \varphi(t)\,,\\
\varphi(0) = \Phi\,,
\end{array}
\right.\,
\]
as the gradient flow of $\mc{F}$ with respect to the $L^2_{\nu}$ metric, starting at $\Phi \in \mathrm{Diff}(\Omega)$.
Denoting $\rho(t) = \varphi(t)_\# \nu$ and $\rho^0 = \Phi_\#\nu$, then $\rho(t)$ statisfies the continuity equation with velocity field $-\nabla U'\left( \ed \rho(t)/\ed \mu \right)$ tangent to the boundary, i.e. 
\[
\left\{ \begin{array}{l}
\displaystyle \partial_t \rho(t) - \mathrm{div}\left( \rho(t) \nabla U'\left( \frac{\ed \rho(t)}{\ed \mu} \right)  \right) = 0\,,\\
\displaystyle \nabla U'\left( \frac{\ed \rho(t)}{\ed \mu} \right)\cdot n_{\partial \Omega} = 0 \,,\\
\rho(0) = \rho^0\,.
\end{array}\right.\,.
\]
In order to link this formulation with the discrete model, for any diffeomorphism $\varphi \in \mathrm{Diff}(\Omega)$, let us denote 
\[
\lambda \coloneqq \frac{\ed \varphi^{-1}_\#\mu}{\ed \mu}\,,
\]
or equivalently $\varphi_\#(\lambda \mu) = \mu$. Then, 
\[
\varphi_\# \nu = \varphi_\# \left[ \frac{1}{\lambda }\frac{\ed\nu}{\ed \mu}  \lambda \mu   \right] = \left(\frac{1}{\lambda }\frac{\ed\nu}{\ed \mu} \right)\circ \varphi^{-1} \mu\,,
\]
which implies
\[
\mc{F}(\varphi) = \int U\left(\frac{1}{\lambda }\frac{\ed\nu}{\ed \mu} \right) \lambda \ed \mu\,.
\]
This suggests the definition of the following regularised energy
\[
\mc{F}_\varepsilon(\varphi) \coloneqq \inf_\lambda \frac{W^2_2(\varphi_\#(\lambda \mu),\mu)}{2\varepsilon} +\int U\left(\frac{1}{\lambda }\frac{\ed\nu}{\ed \mu} \right) \lambda \ed \mu\,.
\]
Let $\mc{T}=\{T_i\}_i$ be a fixed tessellation of $\Omega$, and for any given vector $X=(x_i)_i \in (\mathbb{R}^d)^N$ consider the piece-wise constant map $\varphi^X: \Omega \rightarrow \mathbb{R}^d$ such that $\varphi^X(x) = x_i \in \mathbb{R}^d$ for a.e.\ $x \in T_i$. Then, 
\[
\varphi^X_\# \lambda \mu = \sum_i a_i \delta_{x_i} \quad \text{where} \quad a_i = \int_{T_i} \lambda \ed \mu\,,
\]
and
\begin{equation}\label{eq:Fepslamb}
\mc{F}_\varepsilon(\varphi^X) \coloneqq \inf_\lambda \frac{W^2_2(\sum_i a_i \delta_{x_i},\mu)}{2\varepsilon} +\int U\left(\frac{1}{\lambda }\frac{\ed\nu}{\ed \mu} \right) \lambda \ed \mu\,.
\end{equation}
Optimizing over $\lambda$ we find that for any cell $T_i$ there exist a constant $c_i$ such that almost everywhere on $T_i$,  $\ed \nu /\ed \lambda = c_i \lambda $. This implies that
\[
m_i \coloneqq \int_{T_i} \ed \nu = c_i \int_{T_i} \lambda \ed \mu = c_i a_i\,,
\]
which replaced into \eqref{eq:Fepslamb} gives
\[
\mc{F}_\varepsilon(\varphi^X) \coloneqq \inf_a \frac{W^2_2(\sum_i a_i \delta_{x_i},\mu)}{2\varepsilon} +\sum_i U\left(\frac{m_i}{a_i} \right) a_i\,.
\]
In the case where $\mu = \mathrm{Leb}$ this coincides with $F_\varepsilon(X)$ with $\mathbb{L}_N = \mathbb{T}_N(\Omega)$.

In the case where $\varphi:\Omega\rightarrow \tilde{\Omega} \neq \Omega$ the constraint $\varphi_\#(\lambda \mu)= \mu$ is not appropriate, as in this case we only have
\[
\frac{\ed \varphi_\# \lambda \mu}{\ed \mu} = \mathbf{1}_{\tilde{\Omega}}\,.
\]
Hence, we define the regularized energy by
\[
\mc{F}_\varepsilon(\varphi) \coloneqq \inf_{\lambda, \frac{\ed \eta}{\ed \mu}\leq 1} \frac{W^2_2(\varphi_\#(\lambda \mu), \eta)}{2\varepsilon} +\int U\left(\frac{1}{\lambda }\frac{\ed\nu}{\ed \mu} \right) \lambda \ed \mu\,.
\]
In the case where $\mu =\mathrm{Leb}$, by similar computations as above, we find $\mc{F}_\varepsilon(\varphi^X) = F_\varepsilon(X)$, with $\mathbb{L}_N= \mathbb{T}_N^s(\Omega)$. 

Finally, let us denote by $\mathbb{F}_N \coloneqq \{ \varphi\in {L}^2_\nu (\Omega;\mathbb{R}^d)~:~ \varphi(x)=x_i \in\mathbb{R}^d ~\text{ for a.e. } x \in T_i   \}$ the space of piecewise constant flows on the reference tessellation equipped with the $L^2_\nu$ metric. In both the cases described above, the discrete dynamics in \eqref{eq:gfdiscrete} coincides with the gradient flow
\[
\partial_t {\varphi}^X = - \nabla_{\mathbb{F}_N} \mc{F}_\varepsilon(\varphi^X)\,.
\]

\begin{remark}[Generalizations to models with advected quantites]\label{rem:generalizations} The type of energy regularization considered here can be easily generalized to models where multiple scalar functions and densities are advected by the flow, i.e. to the case where
\[
\mc{F}(\varphi) = \int U\left(a_1 \circ \varphi^{-1}, \ldots, a_n \circ \varphi^{-1}, \frac{\ed \varphi_\# \nu_1}{\ed \mu}, \ldots , \frac{\ed \varphi_\# \nu_m}{\ed \mu} \right)  \ed \mu\,,
\]
where now $U:\mathbb{R}^{n+m} \rightarrow \mathbb{R}$, $a_i:\Omega \rightarrow \mathbb{R}$ and $\nu_j \in\mc{P}(\Omega)$  are given scalar functions and probability measures, respectively. In fact, as before, this can be written as a single function of $\lambda$, since by a change of variables
\[
\mc{F}(\varphi) = \int U\left(a_1, \ldots, a_n, \frac{1}{\lambda} \frac{\ed \nu_1}{\ed \mu}, \ldots , \frac{1}{\lambda} \frac{\ed \nu_m}{\ed \mu} \right)  \lambda \ed \mu\,.
\]
Formally, writing the Hamiltonian equations corresponding to such energies,
\begin{equation}\label{eq:so}
\partial_{tt}^2 \varphi(t) = - \nabla_{\mathbb{F}} \mc{F}(\varphi(t))\,,
\end{equation}
one recovers, with appropriate choices of $U$, a large class of compressible fluid models including, e.g., the thermal shallow water equations or the full compressible Euler equations (see, e.g., \cite{holm1998euler,khesin2021geometric}). 
Then the same discretization strategy described in this section leads naturally to simple Lagrangian schemes for all of these models as well. 
\end{remark}

\section{Optimal transport tools and proof of Proposition \ref{prop:gradient}}\label{sec:otapp}

\subsection{Optimal transport} Given two positive measures $\rho,\mu \in \mc{M}_+(\Omega)$ with fixed total mass $\rho[\Omega] = \mu[\Omega]$, the $L^2$ Wasserstein distance between $\rho$ and $\mu$, is defined via the following minimization problem
\begin{equation}\label{eq:w2}
W^2_2(\rho,\mu) = \min \left \{ \int_{\Omega\times \Omega }|x-y|^2 \ed \gamma(x,y) \,;\quad \gamma \in \Pi(\rho,\mu) \right\} \,,
\end{equation}
where $\Pi(\rho,\mu)$ is the set of coupling plans $\gamma\in\mc{M}_+(\Omega\times \Omega)$  satisfying
\[
\int_{\Omega\times \Omega } \psi(x) \ed \gamma(x,y) = \int_\Omega \psi(x) \ed \rho(x)\,,\quad\int_{\Omega\times \Omega } \psi(y) \ed \gamma(x,y) = \int_\Omega \psi(y) \ed \mu(y)\,,
\]
for all functions $\psi \in C(\Omega)$. Problem \eqref{eq:w2} always admits at least one solution $\gamma$, and we call this optimal transport plan from $\rho$ to $\mu$. 

Semi-discrete optimal transport refers to the case one of the two measures  is discrete and the other  is absolutely continuous. By similar arguments to those used in Section, \eqref{eq:w2} admits a dual formulation which can be expressed in terms of Laguerre tessellations. Suppose that $\mu = \sum_i m_i \delta_{x_i}$ where $X=(x_i)_i\in (\mathbb{R}^d)^N$ and $m_i>0$, and that $\rho$ is absolutely continuous, then
\begin{equation}\label{eq:w2dual}
W^2_2(\rho,\mu) = \max_{w\in \mathbb{R}^N} \sum_i\left( \int_{L_i(X,w)} \left( |x-x_i|^2 -w_i \right) \ed \rho(x) + w_i m_i \right),
\end{equation}
see, e.g., Section 4.1 in \cite{merigot2021optimal}. The maximum is always attained and the maximizer $w\in\mathbb{R}^N$ is related to the optimal plan $\gamma$ by
\[
\int_{\Omega \times \Omega} \psi(x,y) \ed \gamma(x,y) = \sum_i \int_{L_i(X,w)} \psi(x,x_i) \ed \rho(x)\, \quad \forall\, \psi \in C(\Omega\times \Omega)\,.
\]

\subsection{Energy reformulation}\label{sec:reformulation} Let us show the equivalence between \eqref{eq:FepsX} and \eqref{eq:FepsXw2}. Suppose that $X\in \mathbb{R}^d\setminus \Delta_N$. Then, by definition of the $W_2$ distance,
\[
F_\varepsilon(X) \geq  \inf_{a\in \mathbb{R}^N_{> 0},\eta \in \mc{C}} \frac{W_2^2 \Big( \sum_{i} a_i\delta_{x_i}, \eta \Big)}{2\varepsilon} + \sum_i C_i(a_i) \eqqcolon E_\varepsilon(X)\,,
\]
where $F_\varepsilon(X)$ is given by \eqref{eq:FepsX}, and where $\mathcal{C} = \{\mathrm{Leb}\}$ if $\mathbb{L}_N(\Omega)= \mathbb{T}_N(\Omega)$, and $\mathcal{C} = \{ f\ed x\,:\, f:\Omega\rightarrow [0,1] \}$ if $\mathbb{L}_N(\Omega) = \mathbb{T}_N^s(\Omega)$. Using the dual formulation \eqref{eq:w2dual} and exchanging inf and sup we find
\[
E_\varepsilon(X) \geq  \sup_{w\in\mathbb{R}^N} \inf_{a\in \mathbb{R}^N_{> 0},\eta \in \mc{C}} \sum_i \left(\int_{L_i(X,w)} \frac{|x-x_i|^2-w_i}{2\varepsilon} \ed \eta(x) + \frac{w_i}{2\varepsilon} a_i + C_i(a_i) \right)\,.
\]
Optimizing over $\eta$ and $a$, we find that the right-hand side is equal to $D_\varepsilon(X)= F_\varepsilon(X)$ and thereofore $F_\varepsilon(X) = E_\varepsilon(X)$.

\subsection{Proof of Proposition \ref{prop:gradient}}\label{sec:proofgradient}
Let $X,Y\in (\mathbb{R}^d)^N\setminus \Delta_N$ and $\mc{L}\in \mathbb{L}_N(\Omega)$ the optimal tessellation associated with $X$. Then
\[
\begin{aligned}
F_\varepsilon(Y)& \leq \sum_i \int_{L_i}\frac{|x-x_i+x_i-y_i |^2}{2\varepsilon} + C(|L_i|) \\
& \leq F_\varepsilon(X) + \langle G_{\varepsilon}(X), Y-X \rangle_{m^0} + |\Omega| \sum_i\frac{|x_i-y_i|^2}{2\varepsilon}\,,
\end{aligned}
\]
where 
\[
(G_\varepsilon(X))_i \coloneqq \frac{|L_i|}{m^0_i} \frac{x_i- b_i}{\varepsilon}\,, \quad b_i \coloneqq \frac{1}{|L_i|} \int_{L_i} x \, \ed x\,.
\]
This shows that $G_\varepsilon(X) \in \partial^+ F_\varepsilon(X)$, the Fréchet superdifferential of $F_\varepsilon$ at $X$. We now prove that $G_\varepsilon(X)$ is continuous, which is implies that $F_\varepsilon(X)$  is $C^1$ and that $G_\varepsilon(X)$ is its gradient at $x$ with respect to the inner product $\langle\cdot,\cdot\rangle_{m^0}$.  For this, we will use the following expression for $F_\varepsilon(X)$ (shown in Section \ref{sec:reformulation}):
\begin{equation}\label{eq:Fwas}
F_\varepsilon(X) = \min_{a\in\mathbb{R}^N, \eta \in \mathcal{C}} \frac{W^2_2(\eta,\sum_i a_i \delta_{x_i})}{2\varepsilon} + \sum_i C_i(a_i) 
\end{equation}
where $\mc{C}$ is a convex subset of $\mc{M}(\Omega)$: in particular $\mathcal{C} = \{\mathrm{Leb}\}$ if $\mathbb{L}_N(\Omega)= \mathbb{T}_N(\Omega)$, and $\mathcal{C} = \{ f\ed x\,:\, f:\Omega\rightarrow [0,1] \}$ if $\mathbb{L}_N(\Omega) = \mathbb{T}_N^s(\Omega)$. If $X\in (\mathbb{R}^d)^N\setminus \Delta_N$, then problem \eqref{eq:Fwas} admits a unique solution which is linked to the solution $\mathcal{L}(X)= \{L_i(X)\}_i$ of problem \eqref{eq:FepsX} by
\[
\eta(X) = \sum_i \bs{1}_{L_i(X)}\ed x \,,\quad a_i(X) = |L_i(X)|\,.
\]
Since the function minimized in \eqref{eq:Fwas} is continuous with respect to  $X$, $a$ and $\eta$ (with respect to the narrow topology) on the set $\{(a,\eta)\, :\,\sum_i a_i = \eta[\Omega], ~a_i>0 ~ \forall \,i\}$, then the optimal $\eta(X)$ and $a_i(X)$ are continuous functions of $X$ on $(\mathbb{R}^d)^N\setminus \Delta_N$. In particular, given a sequence $(X^n)_n\subset(\mathbb{R}^d)^N\setminus \Delta_N$, such that $X^n =(x^n_i) \rightarrow X\in(\mathbb{R}^d)^N\setminus \Delta_N$ for $n\rightarrow \infty$, we have
\[
\eta(X^n) \rightharpoonup \eta(X) \,, \quad \sum_ia_i(X^n) \delta_{x^n_i} \rightharpoonup \sum_i a_i(X) \delta_{x_i} \,.
\]
Denoting by $\gamma^n \in \mathcal{M}(\mathbb{R}^d,\mathbb{R}^d)$ the optimal transport plan from $\eta(X^n)$ to $\sum_ia_i(X^n) \delta_{x^n_i}$, by the stability of optimal transport plans $\gamma^n \rightharpoonup \gamma$, the optimal plan from $\eta(X)$ to $\sum_ia_i(X) \delta_{x}$.  Now, since $X\in(\mathbb{R}^d)^N\setminus \Delta_N$, we can fix $\varepsilon>0$ sufficiently small, so that that $|x^n_i -x^n_j|\geq 3\varepsilon$ for all $i\neq j$, and $|x_i -x^n_i|\leq \varepsilon$, for all $n\geq n_\varepsilon$ with $n_\varepsilon$ sufficiently large. For any $\delta>0$, let us denote by $B(x_i,\delta)$ the closed ball of radius $\delta$ centered at $x_i$, and consider a continuous function $\psi:\mathbb{R}^d \rightarrow \mathbb{R}$ such that $\psi(x)=1$ for $x\in B(x_i,\varepsilon)$, and $\psi(x)=0$ for $x\in \mathbb{R}^d\setminus B(x_i,2\varepsilon)$. Then, taking $n\geq n_\varepsilon$, 
\[
\int_{L_i(X^n)} x\ed x = \int_{\mathbb{R}^d\times \mathbb{R}^d} x \psi(y) \ed \gamma^n(x,y) \rightarrow \int_{\mathbb{R}^d\times \mathbb{R}^d} x \psi(y) \ed \gamma(x,y) = \int_{L_i(X)} x\ed x\,,
\] 
for $n\rightarrow +\infty$,
which shows that $G_\varepsilon(X)$ is continuous.

\section*{
Data Availability Statement}

The code associated with this article is available at the github repository \url{https://github.com/andnatale/gradient_flows_of_interacting_cells}.

\bibliographystyle{plain}      
\bibliography{refs}   

\end{document}